\documentclass[a4paper]{article}
\usepackage[english]{babel}
\usepackage[utf8]{inputenc}
\usepackage{amsmath,amsthm}
\usepackage{amsfonts,amssymb}

\usepackage[colorlinks=true,linktocpage=true,linkcolor=blue,citecolor=red]{hyperref}
\usepackage[capitalise]{cleveref} 

\usepackage{stmaryrd} 

\usepackage{enumitem}
\setenumerate[1]{label=(\arabic*)} 

\usepackage{footmisc} 

\usepackage{mathtools}

\usepackage{tikz,tikz-cd} 
\usetikzlibrary{arrows}

\newcommand{\wlp}{{\begin{tikzpicture}[x=4.3pt,y=4.3pt] 
\draw (0,0) rectangle (1,1);
\draw (0,0) -- (1,1);
\end{tikzpicture}}}
\newcommand{\alp}{{\begin{tikzpicture}[x=4.3pt,y=4.3pt]
\draw (0,0) -- (0,1) -- (1,1) -- (1,0) -- (0,0);
\draw (0,0.3) -- (0.7,1);
\draw (0.3,0) -- (1,0.7);
\end{tikzpicture}}}

\newcommand{\R}{\mathbb{R}}

\newcommand{\Lb}{\mathbb{L}}
\newcommand{\Rb}{\mathbb{R}}
\newcommand{\Eb}{\mathbb{E}}

\newcommand{\Acal}{\mathcal{A}} 
 
\newcommand{\Ecal}{\mathcal{E}} 
\newcommand{\Rcal}{\mathcal{R}} 
\newcommand{\Tcal}{\mathcal{T}}

\newcommand{\Dcal}{\mathcal{D}}

\newcommand{\Lcal}{\mathcal{L}} 
 
\newcommand{\Wcal}{\mathcal{W}} 
 
\newcommand{\Ccal}{\mathcal{C}} 
 
\newcommand{\Bcal}{\mathcal{B}}

\swapnumbers

\theoremstyle{plain}
\newtheorem{theorem}{Theorem}[section]
\newtheorem{lemma}[theorem]{Lemma}
\newtheorem{prop}[theorem]{Proposition}
\newtheorem*{prop*}{Proposition}
\newtheorem{cor}[theorem]{Corollary}

\theoremstyle{definition}
\newtheorem{definition}[theorem]{Definition}

\newtheorem{remark}[theorem]{Remark}
\newtheorem{example}[theorem]{Example}
\newtheorem{assumption}[theorem]{Assumption}
\newtheorem{construction}[theorem]{Construction}
\newtheorem{notation}[theorem]{Notation}

\newtheorem{proof1}[theorem]{Proof of \cref{th:main_saturation}}

\newcommand{\corner}[1]{\,\tikz[baseline=(todotted.base)]{
        \node[inner sep=1pt,outer sep=0pt] (todotted) {$#1$};
        \draw (todotted.north west) -- (todotted.north east);
        \draw (todotted.south west) -- (todotted.north west);
    }}

\newcommand{\id}{\text{Id}}
\newcommand{\set}{\text{Sets}}

\newcommand{\ano}{a}

\newcommand{\cat}{\text{Cat}}
\newcommand{\cof}{\text{Cof}}

\newcommand{\op}{\text{op}}

\DeclareMathOperator{\Hom}{Hom}

\makeatletter
\newcommand\footnoteref[1]{\protected@xdef\@thefnmark{\ref{#1}}\@footnotemark}
\makeatother

\newcommand{\PC}[3]{\text{#3}^\text{#1}_{#2}}

\begin{document}

\pagestyle{plain}
\title{Combinatorial and accessible weak model categories}

\date{}

\author{Simon Henry}


\maketitle
\vspace{-1cm}

\begin{abstract} In a previous work, we have introduced a weakening of Quillen model categories called weak model categories. They still allow all the usual constructions of model category theory, but are easier to construct and are in some sense better behaved. In this paper we continue to develop their general theory by introducing combinatorial and accessible weak model categories. We give simple necessary and sufficient conditions under which such a weak model category can be extended into a left and/or right semi-model category. As an application, we recover Cisinski-Olschok theory and generalize it to weak and semi-model categories. We also provide general existence theorems for both left and right Bousfield localization of combinatorial and accessible weak model structures, which combined with the results above gives existence results for left and right Bousfield localization of combinatorial and accessible left and right semi-model categories, generalizing previous results of Barwick. Surprisingly, we show that any left or right Bousfield localization of an accessible or combinatorial Quillen model category always exists, without properness assumptions, and is simultaneously both a left and a right semi-model category, without necessarily being a Quillen model category itself.
\end{abstract}

\vspace{-0.2cm}

\renewcommand{\thefootnote}{\fnsymbol{footnote}} 
\footnotetext{\emph{email:} shenry2@uottawa.ca}
\renewcommand{\thefootnote}{\arabic{footnote}}


\tableofcontents

\section{Introduction}

In \cite{henry2018weakmodel} we have introduced a notion of ``weak model category'' which, as its name suggests, is a weakening of the notion of Quillen model category. Its main advantage is that it is considerably easier to endow a given category with a weak model structure than a full Quillen model structure, and the weak version is sufficient for almost all practical purposes.

In this first work, we mostly focused on the constructive aspects of the theory. The present paper will be devoted with the, mostly non-constructive, theory of ``combinatorial'' and ``accessible'' weak model categories, where the underlying category is locally presentable and the factorization systems are combinatorial, i.e. cofibrantly generated by a set of arrows,  or just accessible, i.e. admits a functorial factorization which is accessible, or equivalently are cofibrantly generated by a small category in the sense of R.~Garner's algebraic small object argument \cite{garner2009understanding}.

\bigskip 

We will generalize well known results of the theory of combinatorial Quillen model categories:

\begin{itemize}

\item We will extend the theory of left and right Bousfield localization (\cref{sec:Bousfield}) to accessible weak model categories.

\item We will give simple necessary and sufficient conditions under which a weak model category is a left or a right semi-model category in the sense of \cite{barwick2010left} (previously introduced in \cite{spitzweck2001operads}) in \cref{sec:WMS_with_cylinder}.

\item We will show, in \cref{sec:changing_combinatorial}, that for combinatorial and accessible weak model categories some of these conditions can always be enforced by small modifications of the model structure (called saturation) that do not affect any homotopy theoretic properties.

\item Combining these, we will also obtain results for left and right Bousfield localization of accessible left and right semi-model categories considerably more general than those in \cite{barwick2010left} and \cite{batanin2020Bousfield}. The extra generality being both that we are now able to take localizations of accessible model categories (and not just combinatorial) and that we can consider both left and right localizations of left and right semi-model categories.

\item We will present in \cref{sec:Cisinki_Olschok} a generalized version of Cisinski-Olschok's theory to the setting of accessible weak model categories, which allows one to construct weak, left, right or, in some cases, even Quillen model structures on categories that admit well behaved cylinder functors. This, for example, applies automatically to all tensored and cotensored simplicially enriched categories.

\item In \cref{sec:Two_sided} we study a somewhat surprising phenomenon: some weak model categories have both the necessary properties to be a left semi-model category and a right semi-model category, but can still fail to be Quillen model categories due to a subtle incompatibility between the two structures. In this case the left and right semi-model structures are Quillen equivalent, but they have different classes of weak equivalences. We call such categories two-sided semi-model categories, or simply two-sided model categories. Such examples typically arise when taking Bousfield localization of Quillen model categories that do not satisfy the usual properness assumptions.

\item Finally, the two appendices review many results about constructions on locally presentable categories (\cref{sec:lkp_cat}) as well as combinatorial and accessible weak factorization systems (\cref{sec:apendix_awfs}).

\end{itemize}

\bigskip

I believe that this new point of view of weak model categories is considerably simplifying the theory of model categories in general by clearly separating the relevant homotopy theoretic data (encoded by weak model structures) from some other more technical aspects (saturation properties, existence of strong path and cylinder objects) that are needed to obtain left, right or Quillen model categories.

\bigskip

In a separate paper, \cite{bourke2020fibrant} (joint work with J.~Bourke), we generalize T.~Nikolaus' construction \cite{nikolaus2011algebraic} of the model structure on algebraically fibrant objects, and the ``dual'' model structure by E.~Riehl and M.~Ching on coalgebraically cofibrant objects from \cite{ching2014coalgebraic} to arbitrary accessible weak model categories. This allows for a completely self dual treatment of the two constructions. We also use these constructions to show that every accessible weak model category is connected by a zig-zag of Quillen equivalences to both a left and a right accessible semi-model category and that every every combinatorial weak model category is connected by a zig-zag of Quillen equivalences to a Quillen model category. 

Finally, a paper in preparation \cite{henry2020model} we will use the technology developed in the present paper to provide a more complete solution to M.~Hovey's problem of constructing a ``model category of model categories'' similar to the construction in \cite{barton2020model}. More precisely we will construct several right semi-model structures on the category of premodel categories and left Quillen functors between them whose fibrant objects are various type of model categories (e.g. weak model categories, left semi-model categories, right proper left semi-model categories). If one works with simplicially enriched premodel categories, these can be made into Quillen model structures as in \cite{barton2020model}.

\bigskip

Finally, while the main topic of this paper, combinatorial and accessible weak model categories is deeply non-constructive, some part of the paper hold in constructive framework, in fact in the same sort of very weak predicative framework as discussed in \cite{henry2018weakmodel}.
Most of the non-constructivity is in fact concentrated in the saturation construction of \cref{sec:changing_combinatorial}. The results of \cref{sec:Bousfield} are also non-constructive, mostly because they rely on the saturation constructions and on the general form of the small object argument in locally presentable categories. The appendices are also completely non-constructive. This being said, all the results of section \ref{sec:WMS_with_cylinder} (the recognition theorem for semi-model categories) and of section \ref{sec:Two_sided} on two-sidded model categories and the recognition of Quillen model categories, are constructive. With the exception of \cref{cor:Olschoktheorem}, and of some remarks and comment that uses the saturation construction of \cref{sec:changing_combinatorial} the results of \cref{sec:Cisinki_Olschok} are also constructive.

\section{Preliminaries and terminology}

In this section we give the definition of combinatorial and accessible weak model categories and introduce some terminology that will be used everywhere in the paper. We also briefly recall the theory of weak model categories from \cite{henry2018weakmodel} and prove some general lemmas about these notions.

$\lambda$ and $\kappa$ always denotes \emph{regular} cardinals. We generally use $\kappa$ for an \emph{uncountable} regular cardinal and $\lambda$ for an arbitrary one. A set (or category) is said to be $\lambda$-small if it has cardinality strictly less than $\lambda$.

\begin{definition}\label{def:combinatorialCat}\begin{itemize} 
\item[]
\item A \emph{premodel category} is a complete and cocomplete category with two weak factorization systems respectively called ``(cofibrations, anodyne fibrations)'' and ``(anodyne cofibrations, fibrations)'', such that all anodyne cofibrations are cofibrations, or equivalently all anodyne fibrations are fibrations.

\item A premodel category is said to be \emph{$\lambda$-combinatorial} (resp. \emph{$\lambda$-accessible}) if its underlying category is locally $\lambda$-presentable and both its weak factorization systems are $\lambda$-combinatorial (resp. $\lambda$-accessible) according\footnote{See also Theorem~\ref{th:acc_wfs_eq_def} for an equivalent definition in the accessible case.} to \cref{def:lambda_comb_wfs}.

\item A premodel category is said to be \emph{combinatorial} (resp. \emph{accessible}) if it is $\lambda$-combinatorial (resp. $\lambda$-accessible) for some regular cardinal $\lambda$.

\end{itemize}

\end{definition}

This terminology comes from R.~Barton's PhD thesis \cite{barton2020model}. We both introduced these notions independently, but I latter decided to align the terminology of the present article to his.

Note that a $\lambda$-combinatorial premodel category is also $\lambda$-accessible, and it is $\kappa$-combinatorial for any $\kappa \geqslant \lambda$. Similarly a $\lambda$-accessible premodel category is $\kappa$-accessible for any $\kappa \geqslant \lambda$.

\begin{definition}\label{Def:combinatorialWMS}\begin{itemize}
\item[]
\item A premodel category is said to be a \emph{weak model category} if its cofibrations between cofibrant objects and fibrations between fibrant objects form a weak model structure in the sense of \cite{henry2018weakmodel}.

\item A ($\lambda$-)accessible or ($\lambda$-)combinatorial weak model category is a ($\lambda$-)accessible or ($\lambda$-)combinatorial premodel category which is a weak model category.

\end{itemize}

\end{definition}

\begin{remark}The definition and basic theory of weak model categories from \cite{henry2018weakmodel} will be recalled bellow. We also mention that in \cite{henry2018weakmodel} the definition of weak model categories does not include neither (co)-completeness, nor the existence of fully formed weak factorization systems. In the present paper we only use the more restrictive situation of ``premodel categories that are weak model categories'' in the sense above. To clarify the distinction we will call them \emph{factorization weak model categories}. That is, a factorization weak model category is always complete and cocomplete and its cofibrations and fibrations fits into two weak factorization systems. All combinatorial and accessible weak model categories are factorization weak model categories.
\end{remark}

\begin{definition}\label{def:acyclic_map}
Given a premodel category $\Ccal$ one says that:

\begin{itemize}

\item A morphism is an \emph{acyclic cofibration} if it is a cofibration and if it has the left lifting property against all fibrations between fibrant objects.

\item A morphism is an \emph{acyclic fibration} if it is a fibration and if it has the right lifting property against all cofibrations between cofibrant objects.

\end{itemize}

Note that anodyne (co)fibrations are in particular acyclic (co)fibrations.

\end{definition}

\begin{notation} In diagrams, the arrows belonging to these various classes will be represented as follows:

\[ \begin{array}{|c|c|c|} \hline
\text{Cofibrations:}  & \text{Anodyne cofibrations:} & \text{Acyclic cofibrations:} \\[5pt]
 A \hookrightarrow B & A \displaystyle \overset{\ano}{\hookrightarrow} B  & A \overset{\sim}{\hookrightarrow} B \\[3pt] \hline

\text{Fibrations:} & \text{Anodyne fibrations:} & \text{Acyclic fibrations:}\\[5pt]
X \twoheadrightarrow Y  & X \overset{\ano}{\twoheadrightarrow} Y & X \overset{\sim}{\twoheadrightarrow} Y \\ \hline

\end{array} \]

\end{notation}

\begin{remark}\label{rk:def_WMS}The definition of weak model structure given in \cite{henry2018weakmodel} can be spelled out explicitly (in the special case of premodel categories) as follows: a premodel category is a weak model category if and only if it satisfies the cylinder axiom and path object axiom recalled below.

\begin{itemize}

\item \emph{Cylinder axiom:} Every cofibration $A \hookrightarrow X$ from a cofibrant object to a fibrant objects admits a ``relative strong cylinder object'':

\[ X \coprod_A X \hookrightarrow I_A X \rightarrow X \]

such that $X \coprod_A X \hookrightarrow I_A X$ is a cofibration and the first map $X \overset{\sim}{\hookrightarrow} I_A X$ is an acyclic cofibration.

\item \emph{Path object axiom:} Every fibration $A \twoheadrightarrow X$ from a cofibrant object to a fibrant object admits a ``relative strong path object'':

\[ A \rightarrow P_X A \twoheadrightarrow A \times_X A \]

such that  $P_X A \twoheadrightarrow A \times_X A$ is a fibration and the first map $P_X A \overset{\sim}{\twoheadrightarrow} A$ is an acyclic fibration.

\end{itemize}

Indeed, this is exactly definition 2.1.12 of \cite{henry2018weakmodel} where we have removed the factorization axiom which is included in the fact that we are already working with a premodel category.

\end{remark}

\begin{remark}\label{Rk:WMS_onlydepends_on_core} The definition of weak model category given above does not involve the full data of the two weak factorization systems: it only depends on the cofibrations between cofibrant objects and the fibration between fibrant objects. In fact, the point of view of \cite{henry2018weakmodel} was that the notion of cofibration only really make sense if the domain is cofibrant and the notion of fibration only make sense if the target is fibrant. In particular, everything that has been studied in \cite{henry2018weakmodel} only ever involves the cofibrations between cofibrant objects and fibrations between fibrant objects. Now that we work in the accessible/combinatorial case where we always have fully formed weak factorization systems this restriction is no longer present, but to emphasize the importance they have in the theory we introduce the following terminology: \end{remark}

\begin{definition} Given $\Ccal$ a premodel category:

\begin{itemize}

\item A \emph{core} cofibration is a cofibration with cofibrant domain.

\item A \emph{core} fibration is a fibration with fibrant target.

\item The \emph{homotopical core} or simply the \emph{core} of $\Ccal$ is the full subcategory of fibrant or cofibrant objects endowed with the class of core cofibrations and the class of core fibrations.

\item The \emph{left homotopical core}, or simply the \emph{left core} of $\Ccal$ is the full subcategory of cofibrant objects endowed with the class of core cofibrations and acyclic core cofibrations.

\item The \emph{right homotopical core}, or simply the \emph{right core} of $\Ccal$ is the full subcategory of fibrant objects endowed with the class of fibrations and acyclic core fibrations.

\end{itemize}

\end{definition}

More generally, one will call ``core'' a property or notion defined for premodel categories that only depends on their homotopical core. For example core acyclic fibrations are a core notion as they are the core fibrations with the right lifting property against core cofibrations. It should be clear from \cref{Rk:WMS_onlydepends_on_core} that being a weak model category is a core property: given two premodel categories with equivalent core, then one is a weak model category if and only if the other is also one.

We now recall some elements of the theory of weak model categories from \cite{henry2018weakmodel}. We hope this will convince the reader that it is indeed a reasonable notion, which is close enough to the notion of a Quillen model category.

\begin{construction}\label{intro_localization} We showed in Section 2.1 and 2.2 of \cite{henry2018weakmodel} that in a weak model category as above one can define the homotopy relation between maps from a cofibrant object to a fibrant object, either using a weak\footnote{see \cref{def:weak_cylinder}} cylinder object or a weak path object, and that this is an equivalence relation compatible with composition (and not depending on the choice of the cylinder or path objects). This allows one to directly define the ``homotopy category'' :

\[Ho(\Ccal)\]

of a weak model category $\Ccal$ as the category of fibrant-cofibrant objects of $\Ccal$, with homotopy classes of maps between them. Theorem 2.2.6 of \cite{henry2018weakmodel} shows that this category $Ho(\Ccal)$ is equivalent to various formal localizations:

\begin{itemize}

\item The localization of the category of cofibrant objects at the class of core acyclic cofibrations.

\item The localization of the category of fibrant objects at the class of core acyclic fibrations.

\item The localization of the category $\Ccal^{c \vee f}$ of objects that are either fibrant or cofibrant at the class of core acyclic cofibrations and core acyclic fibrations.

\end{itemize}

In particular, when $\Ccal$ is a weak model category, one can define a notion of ``(weak) equivalence'' in $\Ccal$ for arrows between fibrant or cofibrant objects as the arrows that become invertible in this localization. \end{construction}

We proved in \cite{henry2018weakmodel}:

\begin{theorem} Let $\Ccal$ be a weak model category, $\Ccal^{c \vee f}$ its full subcategory of fibrant or cofibrant objects and $\Wcal$ the class of arrows in $\Ccal^{c \vee f}$ that becomes isomorphisms in the homotopy category, then:

\begin{itemize}

\item $\Wcal$ satisfies the $2$-out-of-$3$ property, the stronger $2$-out-of-$6$ property, is closed under retracts and contains all the isomorphisms.

\item A core cofibration is acyclic if and only if it is in $\Wcal$.

\item A core fibration is acyclic if and only if it is in $\Wcal$.

\end{itemize}
\end{theorem}

\begin{remark}\label{rk:equivalence_implies_WMS} Note that conversely, having a class of equivalences satisfying these three properties implies that $\Ccal$ is a weak model category: if $A \hookrightarrow B$ is a cofibration with $A$ cofibrant and $B$ fibrant, then we factor the codiagonal map:

\[ B \coprod_A B \hookrightarrow I_A B \overset{\ano}{\twoheadrightarrow} B \]

and by $2$-out-of-$3$ for equivalences, both cofibrations $B \hookrightarrow I_A B$ are equivalences, and hence are acyclic cofibrations, so this factorization provides a relative strong cylinder object for $A \hookrightarrow B$. The similar argument applies to path objects.

Moreover, the class of equivalences is the only class having these three properties: Given an arrow $f\colon X \rightarrow Y$ between objects that are either fibrant or cofibrant, we construct a cofibrant replacement $X^{cof}$ of $X$ if $X$ is fibrant, and a fibrant replacement $Y^{fib}$ if $Y$ is cofibrant, and let $f'$ be the composite $f'\colon X^{cof} \rightarrow Y^{fib}$, and consider $f'=p i$ a factorization of $f'$ as a cofibration $i$ followed by an anodyne fibration $p$. Then for any class $W$ having the stability property mentioned above we have:

 \[ f \in W \Leftrightarrow f' \in W \Leftrightarrow i \text{ is an acyclic cofibration.} \]

which shows the uniqueness of the class $W$, as the third condition does not depend on $W$.
\end{remark}

\begin{remark} In the definition of weak model category, we only assumed the existence of ``strong'' path objects and cylinder objects for cofibrations and fibrations from a cofibrant to a fibrant object. In particular this only allows to construct (strong) path objects and cylinder objects for bifibrant objects. However, this is enough to construct a weaker kind of cylinder and path objects for respectively general core cofibrations and core fibrations. These are the ``relative weak path object and cylinder object'' that are introduced in  \cite{henry2018weakmodel} (Definition 2.1.12 and Remark 2.1.13.) which we now recall:
\end{remark}

\begin{definition}\label{def:weak_cylinder} A \emph{relative weak cylinder object} for a core cofibration $i\colon A\hookrightarrow B$ is a diagram of the form:

\[  \begin{tikzcd}
    B \coprod_A B \ar[d,"\nabla"] \ar[r,hook] & I_A B \ar[d] \\
 B \ar[r,hook,"\sim"] & D_A B \\
  \end{tikzcd} \]

where $\nabla$ denote the codiagonal maps and furthermore the first map $B \hookrightarrow I_A B$ is an acyclic cofibration.

\end{definition}

The notion of strong cylinder corresponds to the special case in which the map $B \overset{\sim}{\hookrightarrow} D_A B$ is an isomorphism. One talks about weak cylinder object, for a cofibration of the form $0 \hookrightarrow B$ where $0$ is the initial object, and of course one has a dual notion of weak path objects and relative weak path objects.

It is shown in section 2.1 of \cite{henry2018weakmodel} that we can use weak cylinder and path object instead of the strong ones to define the homotopy relation between maps. 

\begin{construction}\label{constr:rel_weak_strong_cylinder} If a core cofibration $i\colon A \hookrightarrow B$ admits a relative weak cylinder object as in \cref{def:weak_cylinder}, and $B$ is further assumed to be fibrant, then the acyclic cofibration $B \overset{\sim}{\hookrightarrow} D_A B$ will have a section $s$, and the composite: 

\[B \coprod_A B \hookrightarrow I_A B \rightarrow D_A B \overset{s}{\rightarrow} B \]

is a relative strong cylinder object for $i$. Conversely, consider $B \overset{\ano}{\hookrightarrow} B^{fib}$ a fibrant replacement for $B$, if the composite $A \hookrightarrow B^{fib}$ has a relative strong cylinder object $I_A B^{fib}$ then the diagram:

 \[  \begin{tikzcd}
    B \coprod_A B \ar[d,"\nabla"] \ar[r,hook] & I_A B^{fib} \ar[d] \\
 B \ar[r,hook,"\ano"] & B^{fib} \\
  \end{tikzcd} \]

is a relative weak cylinder object. Hence the assumption that any core cofibration has a relative weak cylinder object or that any core cofibration with fibrant target has a relative strong cylinder object are equivalent. \end{construction}

Section 2.3 of \cite{henry2018weakmodel} contains various equivalent definitions of weak model categories (see also Proposition 2.2.10 in \cite{henry2018weakmodel}), some of them seemingly weaker than \ref{rk:def_WMS}, that will be of interest to us. We recall one of them in the following proposition, which has an interesting consequence: for a premodel category $\Ccal$, being a weak model category actually only depends on the left core of $\Ccal$ (i.e. cofibrant objects, core cofibrations and core acyclic cofibrations) or on the right core of $\Ccal$ (i.e. fibrant objects, core fibrations and core acyclic fibrations).

\begin{prop}\label{prop:Model_str_from_3from2} Let $\Ccal$ be a premodel category. Then it is a weak model category if and only if:

\begin{itemize}

\item It satisfies the cylinder axiom, i.e. every core cofibration has a relative weak cylinder object.

\item Acyclic cofibrations satisfies the right cancellation property amongst core cofibration, that is given $i,j$ two composable core cofibrations, if  $j$ and $i \circ j$ are acyclic, then $i$ is acyclic. 

\end{itemize}

Dually, it is also equivalent to the validity of the path object axiom and the left cancellation properties for acyclic fibrations amongst core fibrations.

\end{prop}

Note that because of \cref{lem:2outof6_for_acyclic} below, this second condition really implies the full $2$-out-of-$6$ condition of acyclic cofibrations amongst core cofibrations.

\begin{remark}The notion of weak equivalences and the homotopy category are core notions: that is, if $\Ccal$ admits two weak model structures with the same core, then they have the same weak equivalences and equivalent homotopy categories, with the equivalence of categories being induced by the identity functor.\end{remark}

We briefly mention some lemmas on the notions of acyclic cofibrations and acyclic fibrations:

\begin{lemma}\label{lem:Naiv_when_fibrant_target}
Any acyclic cofibration with fibrant target is an anodyne cofibration. Dually, any acyclic fibration with cofibrant domain is an anodyne fibration.
\end{lemma}

\begin{proof}
If $j \colon A \rightarrow Y$ is an acyclic cofibration and $Y$ is fibrant, if one factors $i$ as $p j'$, with $j'$ an anodyne cofibration followed by a fibration $p$. As $p$ is a fibration with fibrant target, $i$ has the left lifting property against $p$, and hence by the retract lemma $i$ is a retract of $i'$ hence an anodyne cofibration. The same argument applies to the dual case.
\end{proof}

\begin{lemma}\label{lem:2outof6_for_acyclic}
If $k,j,i$ are composable cofibrations such that $kj$ and $ji$ are acyclic then $i$ is acyclic. Dually, if $p,q,r$ are composable fibrations such that $pq$ and $qr$ are acyclic then $p$ is acyclic.
\end{lemma}

A special case of this (corresponding to $k=Id$) is that acyclic cofibrations have the \emph{left cancellation properties amongst cofibrations}, in the sense that if $j,i$ are composable cofibrations such that $ji$ and $j$ are acyclic, then $i$ is acyclic. Dually, \emph{the right cancellation property for acyclic fibrations amongst fibrations} holds.

\begin{proof}

Let $k,j,i$ as above. One needs to show that $i$ has the left lifting property against fibrations between fibrant objects. Starting from the solid diagram:

\[\begin{tikzcd}
A \ar[d,hook,"i"{swap}] \ar[r] & X \ar[d,->>,"p"] \\
B \ar[r] \ar[d,hook,"j"{swap}] & Y \ar[dd,->>] \\
C \ar[d,hook,"k"{swap}] \ar[uur,dotted,"l_2"{description}] &  \\
D \ar[r] \ar[uur,dotted,"l_1"{description}] & 1,
\end{tikzcd} \]

 one constructs $l_1$ using that $ kj$ is acyclic and that $Y$ is fibrant, and then forms the lift $l_2$ using that $ ji$ is acyclic and $p$ is a fibration between fibrant objects. The composite $l_2 j$ produce a diagonal filling for the square hence finishing the proof. The dual proof applies to the other case.
\end{proof}

\begin{remark} Let $i \colon  A \overset{\sim}{\hookrightarrow} B$ be any acyclic cofibration. Let $j\colon  B \overset{\ano}{\hookrightarrow} B^{fib}$ be a fibrant replacement. The composite $ji$ is an acyclic cofibration (by composition) with fibrant target hence it is an anodyne cofibration by \cref{lem:Naiv_when_fibrant_target} as its target is fibrant. It follows that any acyclic cofibration $i$ fits into a composition $j i$ where $j$ and $ji$ are both anodyne cofibrations. It immediately follows that:
\end{remark}

\begin{prop}\label{prop:charac_acyclic_2outof_3}
Acyclic cofibrations are the closure under the left cancellation property amongst cofibrations of anodyne cofibrations. In particular they coincide if and only if anodyne cofibrations have the left cancellation property amongst cofibrations. The dual statement holds for acyclic fibrations.
\end{prop}

The exact same argument also gives that acyclic cofibrations with cofibrant domain are the closure of anodyne cofibrations with cofibrant domain under the left cancellation properties amongst cofibrations with cofibrant domain.

We conclude this preliminary section with the notion of morphism of premodel categories:

\begin{definition}
Given premodel categories $\Ccal$ and $\Dcal$, an adjunction:

\[ L \colon  \Ccal \leftrightarrows \Dcal \colon  R \]

is said to be a Quillen adjunction (or that $L$ is a left Quillen functor or that $R$ is a right Quillen functor) if $L$ preserves cofibrations and $R$ preserves fibrations

\end{definition}

We remind the reader that given an adjunction $L \dashv R$ as above and two weak factorization systems on $\Ccal$ and $\Dcal$ then $L$ preserves the left class if and only $R$ preserves the right class. So the condition above can be equivalently rephrased as the fact that $L$ preserves cofibrations and anodyne cofibration and that $R$ preserves fibrations and anodyne fibrations.

One has the following easy lemma:

\begin{lemma}\label{lem:Qadj_pres_acyclic} Given a Quillen adjunction:

\[L\colon  \Ccal \leftrightarrows \Dcal \colon  R\]

between premodel categories, then $L$ also preserves core cofibrations and acyclic cofibrations and $R$ preserves core fibrations and acyclic fibrations.

\end{lemma}

\begin{proof}
As $L$ preserves cofibrations it also preserves cofibrant objects, hence it preserves cofibrations with cofibrant domain, i.e. core cofibrations. As $R$ dually preserves core fibrations, $L$ preserves the class of maps with the lifting property against core fibrations, hence $L$ preserves acyclic cofibrations. Similarly for $R$.
\end{proof}

\begin{remark} A more general notion of Quillen adjunction (called weak Quillen adjunction) was considered in \cite{henry2018weakmodel}. It should be noted that an adjunction $(L,R)$ where $L$ only preserves core cofibrations and $R$ only preserves core fibrations gives an example of a weak Quillen adjunction. This is enough to obtain an adjunction at the level of the homotopy categories and to show that $L$ preserves core acyclic cofibrations and $R$ preserves core acyclic fibrations. \end{remark}

\section{Recognition principle for semi-model categories}
\label{sec:WMS_with_cylinder}

The notion of left semi-model category has been discovered first by M.~Hovey in \cite{hovey1998monoidal}, who observed that if $A$ is a cofibrant symmetric monoid in a combinatorial\footnote{Hovey works more generally with a notion of ``cofibrantly generated'' that involves smallness conditions.} symmetric monoidal model category, then, even if the monoidal model category $\Ccal$ do not satisfies the ``monoid axiom'', then the category of $A$-algebra still ``almost'' have a model structure (Theorem 3.3 of \cite{hovey1998monoidal}).

The notion has been first explicitely formalized by M.~Spitzweck in his PhD thesis \cite{spitzweck2001operads} under the name $J$-semi-model category. The dual notion of right semi-model category have been first studied by C.~Barwick in \cite{barwick2010left} under the name right (and left) model category.
In Section 12 of \cite{fresse2009modules}, B.~Fresse also introduced a slightly weaker version of the notion that he also calls a left semi-model category.\footnote{This weakening was also very briefly mentioned in Spitzweck's work under the name $(I,J)$-semi-model category.}

In this section, we will talk about ``Spitzweck left/right semi-model categories'' and ``Fresse left/right semi-model category'' to distinguish between the two notions. In the rest of the paper, we will refer to Spitzweck left/right semi-model categories simply as left/right semi-model categories. Precisely:


\begin{definition}\label{left_semi_model} A \emph{Fresse left semi-model category} is a category $\Ccal$ with limits and colimits and three classes of maps called weak equivalences, cofibrations and fibrations such that:

\begin{enumerate}
\item Weak equivalences satisfies $2$-out-of-$3$. All three classes are stable under retract.

\item Any map with cofibrant domain can be factored as a cofibration followed by a trivial\footnote{Here we use ``trivial (co)fibration'' in its usual sense: to mean a (co)fibration which is an equivalence.} fibration, and as a trivial cofibration followed by a fibration.

\item Given a \emph{core} cofibration $i$ and a fibration $p$, one of them being an equivalence, then $i$ has the left lifting property against $p$.

\item Fibrations are stable under pullback.

\item The initial object is cofibrant.

\end{enumerate}

A \emph{Spitzweck left semi-model category} is a Fresse left semi-model category where cofibrations and trivial fibrations form an actual weak factorization system without restrictin to arrows with cofibrant domain.
\end{definition}

We define dually \emph{Fresse right semi-model category} and \emph{Spitzweck right semi-model category}, which are the categorical dual of the above. There the existence of factorizations are restricted to arrows with a fibrant target and lifting properties are restricted to cofibrations and core fibrations.

When not specified, \emph{left/right semi-model categories} will mean \emph{Spitzweck} left/right semi-model category.

\begin{remark} It is important to note that the definition of Fresse left semi-model categories only involves core cofibrations (and fibrations and equivalences). At no point the definition involve or say anything about non-core cofibration. For example if $\Ccal$ is a Fresse left semi-model category then $\Ccal$ endowed with the same fibrations and weak equivalence but only core cofibrations as its class of cofibrations is also a Fresse left semi-model category, and conversely, if $\Ccal$ with only core cofibrations is a Fresse left semi-model category, then $\Ccal$ itself already is.

This does not apply to Spitzweck's definition of left semi-model category which requires (cofibration, trivial fibration) to be a weak factorization system. Though it should be noted that, by the usual retract lemma, a map with cofibrant domain is a cofibration if and only if it has the right lifting property against trivial fibrations. So, if $\Ccal$ is a Fresse left semi-model category, and the trivial fibrations of $\Ccal$ are the right class of a weak factorization system, then if one redefine the class of cofibration as being the left class of this weak factorization, this does not change the core cofibration of $\Ccal$ and makes $\Ccal$ into a Spitzweck left semi-model category. The dual discussion apply to right semi-model categories. \end{remark}

\begin{remark} We will show in the next section that any ``combinatorial'' Fresse left/right semi-model structure is in fact a Spitzweck left semi-model structure in the sense of the previous remark. More precisely, the right saturation construction of \cref{sec:changing_combinatorial} turn any Fresse left semi-model category into a Spitzweck ones (and dually). This is why, in the rest of the paper we will mostly ignore the distinction between Fresse and Spitzweck's definition of left semi-model structure and use the terminology ``left/right semi-model structure'' to mean a Spitzweck left/right semi-model structure.
This being said, outside of the combinatorial and accessible cases, when saturation is no longer available, it seems that Fresse's definition is better behaved as there is no reasons in general for the trivial fibrations to be part of the weak factorization system. In fact, in \cite{henry2020model}, the model structures we will construct on the bicategory of combinatorial premodel categories are Fresse right semi-model structure that, as far as we know, are not Spitzweck right semi-model structure.
\end{remark}

Before moving on to the main theorem of this section, which gives a characterization of left and right semi-model categories amongst factorization weak model categories, we want to show how the usual concepts of model category theory relate to the ones we have introduced in the previous section for premodel categories and weak model categories.

\begin{lemma}\label{lem:acyclic_fib_equiv} Let $\Ccal$ be a Fresse left semi-model category:

\begin{itemize}

\item An arrow that has the right lifting property against all core cofibrations is an equivalence.

\item An arrow between cofibrant objects that has the left lifting property against all core fibrations is an equivalence.

\end{itemize}

The dual properties holds in a right semi-model category.

\end{lemma}

\begin{proof} We first observe that to prove the first point, it is enough to show it for an arrow between cofibrant objects. Indeed let $f\colon  X \rightarrow Y$ an arrow with the right lifting property against core cofibrations. Take cofibrant replacements as follows:

\[\begin{tikzcd}
X^{cof} \ar[r,->>,"\sim"] \ar[dr,"f''"{swap} ] &X' \ar[d,"f'"{swap}] \ar[r,->>,"\sim"] \ar[dr,phantom,"\lrcorner"{very near start}] & X \ar[d,"f"] \\
& Y^{cof} \ar[r,->>,"\sim"] & Y 
\end{tikzcd}\]

The arrow $f'$ has the same lifting property as $f$ because it is a pullback, and hence $f''$ also has the lifting property against core cofibrations because it is a composite of $f'$ with a trivial fibration. By $2$-out-of-$3$ for equivalences, $f$ is an equivalence if and only if $f'$ is, if and only if $f''$ is, hence it is indeed enough to show the result for the arrow $f''$, between cofibrant objects.

Hence one can freely assume that $f\colon X \rightarrow Y$ is an arrow between cofibrant objects. We can hence factor it as a cofibration followed by a trivial fibration. The cofibration has a cofibrant domain, so it has the left lifting property against $f$, it hence follows by the retract lemma that $f$ is a retract of the trivial fibration part of this factorization, and hence is itself an equivalence.

The proof of the second point is the exact dual, except that for the first part involving replacement, and for the construction of the factorization in the second part one needs to assume that the object involved are cofibrant.
\end{proof}

\begin{prop}\label{prop:saturation_of_semi_structure} Let $\Ccal$ be a premodel category which is also a Fresse left semi-model category, with the same fibrations and cofibrations, then:

\begin{enumerate}

\item\label{prop:saturation_of_semi_structure:acyclic_match} A fibration or a core cofibration is acyclic in the sense of the premodel structure on $\Ccal$ if and only if it is trivial.

\item\label{prop:saturation_of_semi_structure:wms} $\Ccal$ is a weak model category.

\item\label{prop:saturation_of_semi_structure:core_left_sat} All core trivial/acyclic cofibrations of $\Ccal$ are anodyne cofibrations.

\item\label{prop:saturation_of_semi_structure:equivalences_match} An arrow between fibrant or cofibrant objects is an equivalence in the sense of the weak model structure if and only if it is an equivalence in the sense of the left semi-model structure.

\item\label{prop:saturation_of_semi_structure:spitzweck_right_sat} If $\Ccal$ is furthermore a Spitzweck left semi-model category, then all trivial/acyclic fibrations are anodyne fibrations.

\end{enumerate}

All the dual properties hold for a right semi-model category.

\end{prop}

\begin{proof} We start with \ref{prop:saturation_of_semi_structure:acyclic_match} for fibrations. Any acyclic fibration is an equivalence by \cref{lem:acyclic_fib_equiv}, hence is also a trivial fibration. Conversely if $f$ is a trivial fibration, i.e. both a fibration and an equivalence, then by definition \cref{left_semi_model} it has the right lifting property against all core cofibration hence is an acyclic fibration.

We move to \ref{prop:saturation_of_semi_structure:acyclic_match} for cofibrations, which we prove together with \ref{prop:saturation_of_semi_structure:core_left_sat}. If $f$ is a core acyclic cofibration, then it is a core cofibration and an equivalence by \cref{lem:acyclic_fib_equiv} and hence it is a trivial cofibration. Conversely, a trivial core cofibration has the left lifting property against all fibrations by \cref{left_semi_model} and hence is both an anodyne cofibration and an acyclic cofibration.

Point \ref{prop:saturation_of_semi_structure:wms} follows immediately as one can construct cylinder and path objects using the factorization system and $2$-out-of-$3$ for equivalences. For example, if $A \hookrightarrow B$ is a core cofibration, then $B \coprod_A B \to B$ can be factored as a cofibration followed by a trivial fibration $B \coprod_A B \hookrightarrow I_A B \overset{\sim}{\twoheadrightarrow} B$, and by $2$-out-of-$3$ for equivalence in the diagram below, the core cofibration $B \hookrightarrow I_A B$ is acyclic. The dual argument produces path objects.
\[
\begin{tikzcd}
 & I_A B \ar[dr,two heads,"\sim"] & \\
  B \ar[ur,hook] \ar[rr,equal] & & B
\end{tikzcd}
\]

For point \ref{prop:saturation_of_semi_structure:equivalences_match}: consider first an arrow $f$ from a cofibrant object to a fibrant object. We factor $f$ as a cofibration followed by a trivial fibration. Then both in the weak model structure and in the left semi-model structure, $f$ is an equivalence if and only if the cofibration part is trivial/acyclic, so the two notions are equivalent. For a general map $f$ between objects that are either fibrant or cofibrant, if the domain is not cofibrant, then it is fibrant so one can pre-compose $f$ with a cofibrant replacement of its domain. Similarly if the target of $f$ is not fibrant, one post-composes $f$ with a fibrant replacement of the target. Both replacement maps are equivalences in both senses, so $f$ is an equivalence in either the weak or the left semi-model structure if and only if the composition with these replacements is an equivalence, but for this composite the result follows from the case previously treated.

For point \ref{prop:saturation_of_semi_structure:spitzweck_right_sat}, if $\Ccal$ is a Spitzweck left semi-model structure, then (cofibration,acyclic fibration) should be a weak factorization system, but as (cofibration,anodyne fibrations) is already assumed to be a weak factorization system in a premodel category the two right classes should coincide.
\end{proof}

We introduce the following terminology:

\begin{definition}\label{def:saturation}
 A premodel category $\Ccal$ is said to be:
\begin{itemize}

\item \emph{Right saturated} if all acyclic fibrations are anodyne fibrations.

\item \emph{Core right saturated} if all core acyclic fibrations are anodyne fibrations.

\item \emph{Left saturated} if all all acyclic cofibrations are anodyne cofibrations.

\item \emph{Core left saturated} if all all core acyclic cofibrations are anodyne cofibrations.

\item \emph{Bi-saturated} (or simply \emph{saturated}) if it is both left and right saturated.



\end{itemize}

\end{definition}

We will show in \cref{sec:changing_combinatorial} that any combinatorial or accessible premodel category can be modified to satisfy any of these properties without changing its core.

We can now state the main theorem of this section:

\begin{theorem}\label{Th:recog_LMS} Let $\Ccal$ be a premodel category. Then there is a class of equivalences making $\Ccal$ into a Fresse left semi-model structure if and only if the following conditions holds:

\begin{enumerate}

\item $\Ccal$ is a weak model category.

\item Every cofibrant object of $\Ccal$ admit a strong cylinder object, i.e. a factorization:

\[ X \coprod X \hookrightarrow IX \rightarrow X \]

where the composition $X \hookrightarrow IX$ of the cofibration with the first coproduct inclusion is an acyclic cofibration.

\item  $\Ccal$ is core left saturated.

\end{enumerate}

when these conditions hold, the class of equivalences is unique.
Moreover, $\Ccal$ is a Spitzweck left semi-model category if and only it is also right saturated.
\end{theorem}

As we have stated the theorem for general premodel categories, one automatically obtains a dual result, which given its importance we will state explicitly:

\begin{theorem}\label{Th:recog_RMS}Let $\Ccal$ be a premodel category.  Then there is a class of equivalences making $\Ccal$ into a Fresse right semi-model structure if and only if the following conditions holds:

\begin{enumerate}

\item $\Ccal$ is a weak model category.

\item Every fibrant object of $\Ccal$ admit a strong path object, i.e. a factorization:

\[ X \rightarrow PX \twoheadrightarrow X \times X \]

where the composition $PX \twoheadrightarrow X$ of the fibration with the first projection is an acyclic fibration.

\item  $\Ccal$ is core right saturated.

\end{enumerate}

When these conditions hold, the class of equivalences is unique. Moreover, $\Ccal$ is a Spitweck right semi-model category if and only if it is also left saturated. \end{theorem}

The rest of the section is devoted to the proof of \cref{Th:recog_LMS}.

At this point it is clear that assuming the existence of a left semi-model structure on $\Ccal$ the three conditions of \cref{Th:recog_LMS} holds: The existence of the weak model structure and the saturation properties have been showed \ref{prop:saturation_of_semi_structure}, and finally, given any cofibrant object $X$, a factorization as a cofibration and anodyne fibration of its codiagonal map:

\[ X \coprod X \hookrightarrow IX \overset{\sim}{\twoheadrightarrow } X \] 

produce a strong cylinder object as the map $X \hookrightarrow IX$ will be acyclic by $2$-out-of-$3$. It is also clear that in this case the left semi-model structure is unique as one can characterize equivalences in terms of the factorization systems.

 \begin{assumption}\label{assumption_2} From now one, we consider a premodel category $\Ccal$ satisfying only the first two conditions of \cref{Th:recog_LMS}, i.e. it is a weak model category where every cofibrant object has a strong cylinder object.

We will construct the class of equivalences so that, up to saturation, it is a left semi-model category. This is achieved by first showing that the definition of the homotopy category of $\Ccal$ as a weak model category can be extended as a localization of the whole category $\Ccal$, instead of localization of its full subcategory of fibrant or cofibrant objects. 
\end{assumption}

\begin{remark}The reason why we do not assume saturation at this point is simply because it is very inessential to the proof and only plays a small role at the end. While this condition is easy to obtain in the combinatorial case (as we will see in the next section), it is much harder to obtain for general premodel categories (or in a constructive mathematics where the result of \cref{sec:changing_combinatorial} are not available), so we believe it is important to make explicit how much can be deduced without this assumptions.\end{remark}

We recall (see \ref{intro_localization}) that the homotopy category of a weak model category can be described as the localization of its full subcategory of cofibrant objects at (core) acyclic cofibrations.

\begin{prop}\label{prop:localization_Strong_cylinder}
Under the \cref{assumption_2}, the functor:

 \[ Ho(\Ccal) \simeq \Ccal^{cof}[(\text{core acyclic cof.})^{-1}] \rightarrow \Ccal[T^{-1}] \]

is an equivalence, where $T$ is the class of all acyclic fibrations and all core acylic cofibrations.
\end{prop}

\begin{proof} We apply (the dual of) Lemma 2.2.5 of \cite{henry2018weakmodel}, with $\Ccal \colon  =\Ccal$, $\Dcal \colon = \Ccal^{cof} \subset \Ccal$, $W'$ the class of core acylic cofibrations and $W$ the class of acyclic fibrations. We check all the conditions:

\begin{enumerate}

\item $\Ccal^{cof}[(W')^{-1}]$ exists and is the homotopy category of $\Ccal$.

\item Given an object $A \in \Ccal$, one has a cofibrant replacement $A^c \overset{\ano}{\twoheadrightarrow} A$ by factoring the map $\emptyset \rightarrow A$ as a cofibration followed by an anodyne fibration.

\item Given $A \in \Ccal$, $C,C'$ two cofibrant objects of $\Acal$ and $f\colon C \rightarrow A$ an arrow in $W$, i.e. an acyclic fibration, then for any solid diagram:

\[
\begin{tikzcd}
  & C \ar[d,"f"] \\
C' \ar[ur,dotted,"v"] \ar[r,"g"] & A
\end{tikzcd}
\]

there is a dotted arrow $v$ making the triangle commute: it is obtained by the lifting property of $C'$ against $f$.

\item Given $v$ and $v'$ two lifts as above, one can form the solid diagram, and its diagonal filler:

\[
\begin{tikzcd}
C' \coprod C' \ar[d,hook] \ar[rr,"(v;v')"] &  &  C \ar[d,"f"] \\
I C' \ar[r] \ar[urr,dotted,"h"{description}] & C' \ar[r,"g"] & A 
\end{tikzcd}
\]

$h$ proves that $v$ and $v'$ are equal in the homotopy category $Ho(\Ccal) =\Ccal^{cof}[(W')^{-1}]$.

\item $W$ is stable under composition. \end{enumerate} \end{proof}

\begin{definition}\label{def:StrongCylequiv} One says that a morphism in $\Ccal$ is an equivalence if it is invertible in the localization of \ref{prop:localization_Strong_cylinder}. \end{definition}

\begin{remark}\label{rk:acyclic_fib_are_eq} It is immediate from the definition that acyclic fibrations and core acyclic cofibrations are equivalences.
\end{remark}

\begin{remark}\label{rk:compatibility_of_eq} The functor $\Ccal^{c \vee f}[V^{-1}] \rightarrow \Ccal[T^{-1}]$ where $T$ is as in \cref{prop:localization_Strong_cylinder} and $V$ is the class all core acyclic (co)fibrations is an equivalence as both are equivalent to $\Ccal^{cof}[(\text{Acyclic cof.})^{-1}]$ in a compatible way. Indeed one equivalence is proved in \cref{prop:localization_Strong_cylinder} and the other is Theorem 2.2.6 of \cite{henry2018weakmodel}, which was proved in the exact same way.

This immediately implies the compatibility of our new definition of equivalences with the equivalences of the weak model structure as, in both cases, the equivalences are the maps invertible in these equivalent localizations. In particular a core cofibration or a core fibration is acyclic if and only if it is an equivalence. Assuming core left saturation we can improve this to:
\end{remark}

\begin{prop}\label{prop:tractable_acyc_fib} If $\Ccal$ satisfies \cref{assumption_2} and is core left saturated, then a fibration in $\Ccal$ is an equivalence in the sense of \cref{def:StrongCylequiv} if and only if it is an acylic fibration.\end{prop}

\begin{proof} That fact that acyclic fibration are equivalences follows immediately from \cref{def:StrongCylequiv}. We need to show that conversely a fibration which is an equivalence is acyclic. We first prove it for a fibration $f\colon X \rightarrow Y$ with cofibrant domain. We factor $f$ as a (core) cofibration followed by an anodyne fibration. The anodyne fibration is acyclic, and hence is an equivalence. It follows by $2$-out-of-$3$ that the core cofibration is an equivalence and hence acyclic as mentioned in \cref{rk:compatibility_of_eq}. In particular by the saturation assumptions on $\Ccal$, this core acyclic cofibration is an anodyne cofibration, i.e. it has the left lifting property against all fibration, including $f$, and it follows by the usual retract lemma that $f$ is a retract of the anodyne fibration part of the factorization, hence an anodyne fibration itself.

For the case of a fibration $f\colon X \rightarrow Y$ with not necessarily cofibrant domain $X$, one takes a cofibrant replacement $p \colon X^{cof} \overset{\sim}{\twoheadrightarrow} X$ of the domain and uses that $f$ is a fibration and both $p$ and $f p$ are acyclic fibrations by the first half of the proof to conclude that $f$ is an acyclic fibration (\cref{lem:2outof6_for_acyclic}).
\end{proof}

This concludes the proof of \cref{Th:recog_LMS}: One direction, as well as the uniqueness of the class of equivalence has been discussed before. Conversely, we have showed that under the assumption of \cref{Th:recog_LMS} we have a class of equivalences satisfying $2$-out-of-$3$ and stable under retract (in fact, satisfying $2$-out-of-$6$), such that a fibration or a core cofibration is acyclic if and only if it is an equivalence. We can hence deduce all the conditions of \cref{left_semi_model} by simply using the factorization and lifting property in premodel categories.

For the case of Spitzweck right semi-model categories, we have already noticed that they are right saturated in \cref{prop:saturation_of_semi_structure}.\ref{prop:saturation_of_semi_structure:spitzweck_right_sat}, and conversely, if we also assume that $\Ccal$ is right saturated then (cofibration, acyclic fibration) do form a weak factorization system as it coincides with the (cofibration,anodyne fibration) weak factorization system.

\section{Saturation}
\label{sec:changing_combinatorial}

The goal of this section is to study how any combinatorial or accessible premodel category can be modified so that it satisfies any of the saturation properties defined in \ref{def:saturation}, without modifying its core, and preserving its combinatorial or accessible character.

In my opinion these are the results that really give strength to \cref{Th:recog_LMS,Th:recog_RMS}, and the reason why these theorems are considerably more interesting for combinatorial and accessible categories than for general premodel categories. Namely, they show that the condition for being a semi-model category can be split into two distinct parts. Firstly, being a weak model structure with strong cylinder objects, which are conditions only involving the core of the premodel structure and are often very easy to obtain from an enrichment or a monoidal structure as shown in \cref{sec:Cisinki_Olschok}. Secondly, saturation conditions that can always be imposed by modifying the (combinatorial or accessible) premodel structure without changing the core, i.e. without affecting the validity of the first set of conditions nor the resulting homotopy theory.

The main results of this section can be summarized in the following theorem:

\begin{theorem}\label{th:main_saturation} In the category of accessible premodel categories and left Quillen functors between them:

\begin{itemize}

\item The full subcategories of left saturated and core left saturated accessible premodel categories are reflective, with reflection respectively denoted $\Ccal \rightarrow \Lb \Ccal$ and $\Ccal \rightarrow \Lb^c \Ccal$.

\item The full subcategories of right saturated and core right saturated accessible premodel categories are co-reflective, with coreflection respectively denoted $\R \Ccal \rightarrow \Ccal$ and $\R^c \Ccal \rightarrow \Ccal$.

\end{itemize}

These premodel categories $\Lb \Ccal, \Lb^c \Ccal, \R \Ccal$ and $\R^c \Ccal$ all have the same underlying category and the same core as $\Ccal$. $\Lb \Ccal$ and $\Lb^c \Ccal$ also have the same cofibrations, acyclic cofibrations and anodyne fibrations as $\Ccal$ and $\R \Ccal$ and $\R^c \Ccal$ have the same anodyne cofibrations, fibrations and acyclic fibrations as $\Ccal$.

Moreover, if $\Ccal$ is $\kappa$-combinatorial (resp. $\kappa$-accessible) for $\kappa$ an uncountable regular cardinal, then they are all $\kappa$-combinatorial (resp. $\kappa$-accessible) as well\footnote{In particular, the theorem still holds if one replaces accessible by combinatorial everywhere.}. Finally, if $\Ccal$ satisfies any of the saturation properties of \cref{def:saturation} then they all satisfy the same property.

\end{theorem}

\begin{remark} In particular, if $\Ccal$ is a combinatorial or accessible weak model category, all the left Quillen functors  $\R \Ccal, \R^c \Ccal \to \Ccal \to \Lb \Ccal, \Lb \Ccal^c$ are Quillen equivalences as they induce equivalences of categories between the cores, and only cores are involved in computing homotopy categories. They also have the same weak equivalences (between fibrant or cofibrant objects).
\end{remark}

Explicit descriptions of these saturation constructions are given by \cref{lem:saturation_1} and \cref{Constr:Saturation}. Of course if one chooses to work instead in the category whose morphism are right Quillen functors, then the right saturation constructions $\R$ and $\R^c$  become reflections while the left saturation constructions $\Lb$ and $\Lb^c$ become co-reflections. An immediate corollary of the theorem is:

\begin{cor} Given an accessible (resp. combinatorial) premodel category $\Ccal$ there exists a bi-saturated accessible (resp. combinatorial) premodel structure on $\Ccal$ with the same core as $\Ccal$. \end{cor}
\begin{proof} By \cref{th:main_saturation}, both $\R \Lb \Ccal$ or $\Lb \R \Ccal$ have this property. \end{proof}

\begin{remark} I expect $\R \Lb \Ccal$ and $\Lb \R \Ccal$ to be different in general. For example, a quick computation using the explicit construction below shows that the class of cofibrations of $\Lb \R \Ccal$ is generated by the core cofibrations of $\Ccal$ together with the \emph{anodyne} cofibrations of $\Ccal$, while the class of cofibrations of $\R \Lb \Ccal$ is generated by the core cofibrations of $\Ccal$ together with the \emph{acylic} cofibrations of $\Ccal$. I do not see any reasons for these two classes to be the same in general, though I admit that I did not look for an explicit counterexample.
\end{remark}

We need a few lemmas before one can prove \cref{th:main_saturation}.

\begin{lemma}\label{lem:saturation_1} Let $\Ccal$ be a premodel category, and let $S$ be a class of acyclic cofibrations in $\Ccal$. We assume that there exists a weak factorization system on $\Ccal$ whose right class are the fibrations of $\Ccal$ with the right lifting property against all maps in $S$. And let $\Lcal^S \Ccal$ be the premodel category structure on $\Ccal$ with the same cofibrations and anodyne fibrations as $\Ccal$ and where the anodyne cofibrations and fibration are given by this new weak factorization system. Then:

\begin{enumerate}

\item\label{lem:saturation_1:samecore} $\Lcal^S \Ccal$ also has the same core and the same acyclic cofibrations as $\Ccal$.

\item\label{lem:saturation_1:universality} The identity $\Ccal \rightarrow \Lcal^S \Ccal$ is a left Quillen functor and is universal for left Quillen functors $\Ccal \rightarrow \Dcal$ sending all maps in $S$ to anodyne cofibrations.

\item\label{lem:saturation_1:rightsaturation} If $\Ccal$ is right saturated or core right saturated then $\Lcal^S \Ccal$ has the same property.

\end{enumerate}

\end{lemma}

One can dually consider, if it exists, the premodel structure $\Rcal^K \Ccal$ for $K$ a class of acyclic fibrations, which has all the dual properties.

\begin{proof}
\begin{enumerate}
\item By construction, $\Lcal^S \Ccal$ has the same cofibrations as $\Ccal$. They also have the same core fibrations: any fibration of $\Lcal^S \Ccal $ is in particular a fibration of $\Ccal$, and a core fibration in $\Ccal$ has the right lifting property against all maps in $S$ because they are assumed to be acyclic cofibrations, so it is a fibration of $\Lcal^S \Ccal$. In particular, having the same cofibrations and core fibrations as $\Ccal$, $\Lcal^S \Ccal$ has the same acyclic cofibrations as well.

\item All cofibrations and anodyne cofibrations of $\Ccal$ are cofibrations and anodyne cofibrations in $\Lcal^S \Ccal$ and a left Quillen functor $F\colon  \Ccal \rightarrow \Dcal$ is also a left Quillen functor $\Lcal^S \Ccal \rightarrow \Dcal$ if and only every fibration in $\Dcal$ is sent by the right adjoint to a fibration in $\Lcal^S \Ccal$, i.e. has the right lifting property against all maps in $S$, which happens if and only if all arrows in $S$ are sent to anodyne cofibration by $L$.

\item If one assumes that $\Ccal$ is (core) saturated, and $f$ is a (core) acyclic fibration in $\Lcal^S \Ccal$, then $f$ is also a (core) acyclic fibration in $\Ccal$ (for example because $\Lcal^S \Ccal \rightarrow \Ccal$ is a right Quillen functor and \cref{lem:Qadj_pres_acyclic}). Hence by (core) right saturation of $\Ccal$ it is an anodyne fibration in $\Ccal$. As $\Lcal^S \Ccal$ and $\Ccal$ have the same anodyne fibrations, $f$ is also an anodyne fibration in $\Lcal^S \Ccal$.
\end{enumerate}
\end{proof}

\begin{construction}\label{Constr:Saturation} Given an accessible or combinatorial category $\Ccal$, all the reflections and coreflections claimed in \cref{th:main_saturation} are obtained as special cases of the construction described in \cref{lem:saturation_1}. Explicitly:

\begin{itemize}

\item $\Lb \Ccal$ is $\Lcal^S \Ccal$ for $S$ the class of all acyclic cofibrations.

\item $\Lb^c \Ccal$ is $\Lcal^S \Ccal$ for $S$ the class of all core acyclic cofibrations.

\item $\R \Ccal$ is $\Rcal^K \Ccal$ for $K$ the class of all acyclic fibrations.

\item $\R^c \Ccal$ is $\Rcal^K \Ccal$ for $K$ the class of all core acyclic fibrations.

\end{itemize}

The fact that the corresponding weak factorization systems exist and are appropriately $(\kappa$-)accessible or ($\kappa$-)combinatorial will be proved in \ref{th:main_saturation:proof} by iterated appplications of the following lemma:

\end{construction}

\begin{lemma}\label{lem:Constructing_some_WFS} Fix $\kappa$ an uncountable regular cardinal. Let $(L_1,R_1)$  and $(L_2,R_2)$ two $\kappa$-accessible weak factorization systems on a locally $\kappa$-presentable category $\Ccal$ such that $L_1 \subset L_2$ (or $R_2 \subset R_1$).

\begin{enumerate}

\item\label{lem:Constructing_some_WFS:1} There is a $\kappa$-accessible weak factorization system $(L_3,R_3)$ on $\Ccal$ such that $R_3$ is the class of maps that have the right lifting property against all $L_1$-maps whose domain is $L_2$-cofibrant. If $(L_1,R_1)$ is $\kappa$-combinatorial, then $(L_3,R_3)$ is $\kappa$-combinatorial.

\item\label{lem:Constructing_some_WFS:2} There is a $\kappa$-accessible weak factorization system $(L_4,R_4)$ on $\Ccal$ such that $L_4$ is the class of maps that have the left lifting property against all $R_2$-maps whose target is $R_1$-fibrant. If $(L_2,R_2)$ is $\kappa$-combinatorial, then $(L_4,R_4)$ is $\kappa$-combinatorial.

\end{enumerate}
\end{lemma}

\begin{proof}
For \ref{lem:Constructing_some_WFS:1}, let $L_2$-Cof be the category of coalgebras in $\Ccal$ for a $\kappa$-accessible $L_2$-cofibrant replacement copointed endofunctor (such an endofunctor exists by point \ref{th:acc_wfs_eq_def:accessible} of \cref{th:acc_wfs_eq_def}). By point \ref{th:pres_rank:(co)algebras} of \cref{th:pres_rank}, the category $L_2$-Cof is locally $\kappa$-presentable and the forgetful functor $V\colon L_2$-Cof $\rightarrow \Ccal$ is a $\kappa$-left adjoint in the sense of \cref{not:kappa_adjunction}. We then consider $(L'_1,R'_1)$ the left transfer of $(L_1,R_1)$ to $L_2$-Cof,  that is $L'_1$ is the class of morphisms whose image by the forgetful functor to $\Ccal$ is in $L_1$. It exists and is $\kappa$-accessible (and $\kappa$-combinatorial if $(L_1,R_1)$ is) by \cref{th:appendix_Main}.\ref{th:appendix_Main:left_transfer}. Finally, we define $(L_3,R_3)$ to be the right transfer of $(L'_1,R'_1)$, which also exists and is $\kappa$-accessible (or $\kappa$-combinatorial) by \cref{th:appendix_Main}.\ref{th:appendix_Main:right_transfer}. A map is in $R_3$ if and only if it has the left lifting property against the images by the forgetful functor of all of the maps in $L'_1$. To conclude, we show that this is exactly the class of $L_1$-maps with $L_2$-cofibrant domain. Indeed given an $L_1$-map $f\colon X \rightarrow Y$ with $\emptyset \rightarrow X \in L_2$, then as $f$ is an $L_2$-cofibration one can put an $L_2$-Cof structure on $X$ and $Y$ which makes $f$ into a morphism in $L_2$-Cof (see for example Lemma 2.12 in \cite{garner2020lifting}) and this concludes the proof.

The proof of \ref{lem:Constructing_some_WFS:2} is dual: we consider instead the category $R_1$-Fib of algebras for the pointed $R_1$-fibrant replacement endofunctor (which is a locally $\kappa$-presentable category with a $\kappa$-adjunction with $\Ccal$ also by \cref{th:pres_rank}.\ref{th:pres_rank:(co)algebras}), and we right transfer $(L_2,R_2)$ to $R_1$-Fib and then left transfer it back to $\Ccal$, to get also a $\kappa$-accessible ($\kappa$-combinatorial if $(L_2,R_2)$ is) weak factorization system on $\Ccal$ for the same reason as above. The verification that the weak factorization system $(L_4,R_4)$ obtained in this way has the property claimed in the lemma is exactly as above (it does not involve accessibility conditions so it can be dualized).

\end{proof}

\begin{proof1}\label{th:main_saturation:proof} One first shows that $\Lb \Ccal, \Lb^c \Ccal, \R \Ccal, \R^c \Ccal$ as defined in \cref{Constr:Saturation} all exist and are $\kappa$-accessible or $\kappa$-combinatorial as soon as $\Ccal$ is.

\begin{itemize}

\item The existence of $\Lcal^S \Ccal$ when $S$ is all acyclic cofibrations: We first apply \cref{lem:Constructing_some_WFS}.\ref{lem:Constructing_some_WFS:2} with $(L_1,R_1)=(L_2,R_2) = $(anodyne cofibrations, fibrations).  This gives a weak factorization system whose left class is the class of maps with the left lifting property against core fibrations. We then take the intersection of this with the class of all cofibrations and obtain a weak factorization system whose left class is the class of all acyclic cofibrations (elements of the right class are automatically fibrations). This weak factorization exists and is $\kappa$-accessible (and $\kappa$-combinatorial if $\Ccal$ is) by \cref{th:appendix_Main}.\ref{th:appendix_Main:inf_wfs}.

\item The existence of $\Lcal^S \Ccal$ when $S$ is all core acyclic cofibrations: We first apply \cref{lem:Constructing_some_WFS}.\ref{lem:Constructing_some_WFS:1} to the weak factorization system of the previous point to get a weak factorization system whose right class are the maps with the right lifting property against all core acyclic cofibrations. Using suprema of weak factorization system (as in \cref{nota:before_main_th}) we obtain a weak factorization system whose right class are the fibrations with the lifting property against all acyclic core cofibrations. By \cref{th:appendix_Main}.\ref{th:appendix_Main:sup_wfs} it exists and is $\kappa$-accessible (and $\kappa$-combinatorial as soon as $\Ccal$ is).

\item The existence of $\Rcal^K \Ccal$ when $K$ is the class of all acyclic fibrations or of core acyclic fibrations follows from the exact dual argument exchanging the role of points \ref{lem:Constructing_some_WFS:1} and \ref{lem:Constructing_some_WFS:2} of \cref{lem:Constructing_some_WFS}, and the role of points \ref{th:appendix_Main:sup_wfs} and \ref{th:appendix_Main:inf_wfs} of \cref{th:appendix_Main}.

\end{itemize}

$\Lb \Ccal$ and $\Lb^c \Ccal$ as constructed above have the same cofibrations and same core as $\Ccal$ by \cref{lem:saturation_1}, hence the same acyclic cofibrations as $\Ccal$. As all acyclic cofibrations (resp. core acyclic cofibrations) of $\Ccal$ have been made into anodyne cofibrations in $\Lb \Ccal$ (resp. $\Lb^c \Ccal$) it follows that $\Lb \Ccal$ is left saturated and $\Lb^c \Ccal$ is core left saturated. The dual argument shows that $\R \Ccal$ and $\R^c \Ccal$ are respectively right saturated and core right saturated.

Any left Quillen functor $\Ccal \rightarrow \Dcal$ where $\Dcal$ is (core) left saturated sends all (core) acyclic cofibration to (core) acyclic cofibrations by \cref{lem:Qadj_pres_acyclic}, hence by saturation of $\Dcal$ to anodyne cofibration, and hence by \cref{lem:saturation_1}.\ref{lem:saturation_1:universality} factors (uniquely) through the identity functor $\Ccal \rightarrow \Lb \Ccal$ (or $\Ccal \rightarrow \Lb^c \Ccal$). The dual argument (involving right Quillen functor) gives the universality of $\R \Ccal$ and $\R^c \Ccal$.

Finally the fact that the saturation properties of $\Ccal$ implies the same saturation property of $\Lb \Ccal, \R \Ccal$, etc. follows from \cref{lem:saturation_1}.\ref{lem:saturation_1:rightsaturation} and its dual version. \qed
\end{proof1}

\begin{remark} A different argument, that we will not detail here, shows that if $\Ccal$ is $\omega$-combinatorial then $\R \Ccal$ is $\omega$-combinatorial, but it seems unlikely that this can be extended to any of the other cases. 
\end{remark}

\section{Two-sided semi-model categories}
\label{sec:Two_sided}

Something that came quite as a surprise to the author is that a premodel category can be both a (Spitzweck) left and right semi-model category at the same time, i.e. satisfies all the conditions of \cref{Th:recog_LMS,Th:recog_RMS}, without automatically being a Quillen model category.

\begin{definition}\label{def:twosided} A premodel category $\Ccal$ is said to be a \emph{two-sided semi-model category}, or more simply a \emph{two-sided model category} if:

\begin{itemize}

\item $\Ccal$ is a weak model category.

\item Every fibrant object in $\Ccal$ admits a strong path object and every cofibrant object in $\Ccal$ admits a strong cylinder object.

\item $\Ccal$ is bi-saturated.

\end{itemize}
 
\end{definition}

For example, as we will see in the next section, any bisaturated premodel category wich admits a strong Quillen cylinder as in \cref{def:Strong_Quillen_cylinder} is a two-sided semi-model category.

By \cref{Th:recog_LMS,Th:recog_RMS} a two-sided model category admits both a class of equivalences making it into a left semi-model category and a class of equivalences making it into a right semi-model category. The reason it is not automatically a Quillen model category is that these two classes might not coincide.

\begin{remark}\label{rk:Prop_of_two_sided} Given $\Ccal$ a two-sided semi-model category, we denote by $\Wcal_R$ (resp. $\Wcal_L$) and call ``\emph{right equivalences}'' (resp. ``\emph{left equivalences}''), the class of equivalences making $\Ccal$ into a right (resp. left) semi-model category according to \cref{Th:recog_RMS} (resp. \cref{Th:recog_LMS}). The properties of these classes are summarized below:

\begin{enumerate}

\item\label{rk:Prop_of_two_sided:Core_eq} For arrows between objects that are either fibrant or cofibrant, $\Wcal_R$ and $\Wcal_L$ agree and coincide with the equivalences of $\Ccal$ as a weak model category.

\item\label{rk:Prop_of_two_sided:2outof6} Both classes satisfies $2$-out-of-$6$ and are stable under retracts.

\item\label{rk:Prop_of_two_sideed:L_seeFib} A fibration is anodyne (equivalently acyclic) if and only if it is in $\Wcal_L$ whilst a cofibration is anodyne (equivalently acyclic) if and only if it is in $\Wcal_R$.

\item\label{rk:Prop_of_two_sideed:Two_Loc} The two localizations $\Ccal[\Wcal_R^{-1}]$ and $\Ccal[\Wcal_L^{-1}]$ are equivalent to $Ho(\Ccal) \simeq \Ccal^{c \vee f}[\Wcal^{-1}]$, but this identification gives rise to two different functors $\Ccal \rightrightarrows Ho(\Ccal)$ called respectively the ``right localization'' and the ``left localization functor''. The right localization functor sends each object to a fibrant replacement, while the left localization functor sends each object to a cofibrant replacement.

\item\label{rk:Prop_of_two_sideed:comp_Two_Loc} There is a natural transformation from the left localization functor to the right localization functor given objectwise by the composite $X^{cof} \overset{\ano}{\twoheadrightarrow} X \overset{\ano}{\hookrightarrow} X^{fib} $. Its components are isomorphisms at objects that are either fibrant or cofibrant, but not necessarily for general objects.

\item\label{rk:Prop_of_two_sideed:Eq=local} $\Wcal_R$ is the class of arrows inverted by the right localization functor and $\Wcal_L$ is the class of arrows inverted by the left localization functor.

\end{enumerate}

\end{remark}

\begin{prop}\label{prop:QuillenCharac} Given a two-sided semi-model category $\Ccal$ the following conditions, as well as all their duals, are all equivalent:

\begin{enumerate}

\item\label{prop:QuillenCharac:R=L} $\Wcal_R = \Wcal_L$.

\item\label{prop:QuillenCharac:L_is_Quillen} $\Wcal_L$ makes $\Ccal$ into a Quillen model category.

\item\label{prop:QuillenCharac:trivCof_are_in_L} Every anodyne cofibration is in $\Wcal_L$.

\item\label{prop:QuillenCharac:FibRep} For every arrow $v\colon X \rightarrow Y$ there exists a square:

  \[ \begin{tikzcd}
    X \ar[r,"\in \Wcal_L"] \ar[d,"v"] & X' \ar[d] \\
   Y \ar[r, "\in \Wcal_L"] & Y' \\
  \end{tikzcd} \]

where $X'$ and $Y'$ are fibrant.

\item\label{prop:QuillenCharac:Replace_are_compatible} For every object $X \in \Ccal$, there are some choices of fibrant and cofibrant replacements, such that the composite: 

\[ X^{cof} \overset{\ano}{\twoheadrightarrow} X \overset{\ano}{\hookrightarrow} X^{fib} \]

is an equivalence of $\Ccal$ (seen as weak model category).

\item\label{prop:QuillenCharac:loc_are_equiv} The two localization functors $\Ccal \rightrightarrows Ho(\Ccal)$ (see \cref{rk:Prop_of_two_sided}.\ref{rk:Prop_of_two_sideed:Two_Loc}) are isomorphic.

\end{enumerate}

In particular, if all the objects of $\Ccal$ are fibrant or cofibrant, then all these conditions are satisfied.

\end{prop}

\begin{proof}We first show that all conditions except \ref{prop:QuillenCharac:FibRep} are equivalent:

\begin{itemize}
\item \ref{prop:QuillenCharac:trivCof_are_in_L} implies \ref{prop:QuillenCharac:Replace_are_compatible} because it implies that both $X^{cof} \overset{\ano}{\twoheadrightarrow} X$ and $X \overset{\ano}{\hookrightarrow} X^{fib}$ are in $\Wcal_L$, and hence that the composed map is in $\Wcal_L$, which implies it is an equivalence by \cref{rk:Prop_of_two_sided}.\ref{rk:Prop_of_two_sided:Core_eq}. 
\item \ref{prop:QuillenCharac:Replace_are_compatible} implies \ref{prop:QuillenCharac:loc_are_equiv} because of \cref{rk:Prop_of_two_sided}.\ref{rk:Prop_of_two_sideed:comp_Two_Loc}.
\item \ref{prop:QuillenCharac:loc_are_equiv} implies \ref{prop:QuillenCharac:R=L} because of \cref{rk:Prop_of_two_sided}.\ref{rk:Prop_of_two_sideed:Eq=local}.
\item \ref{prop:QuillenCharac:R=L} implies \ref{prop:QuillenCharac:L_is_Quillen} which in turn implies \ref{prop:QuillenCharac:trivCof_are_in_L} essentially by definition of Quillen model categories (and \cref{rk:Prop_of_two_sided}.\ref{rk:Prop_of_two_sided:Core_eq} and \ref{rk:Prop_of_two_sided:2outof6}).
\end{itemize}
As several of these conditions are clearly self dual, they are also all equivalent to their duals. Finally, Condition \ref{prop:QuillenCharac:trivCof_are_in_L} implies Condition \ref{prop:QuillenCharac:FibRep} simply by using the (anodyne cofibration,fibration) factorization to get the fibrant replacement. The last implication is essentially (a part of) Quillen's path object argument for existence of transfer: we assume Condition~\ref{prop:QuillenCharac:FibRep}, let $A \overset{\ano}{\hookrightarrow} B$ be an anodyne cofibration and choose a fibrant replacement according to Condition \ref{prop:QuillenCharac:FibRep}:

\[
\begin{tikzcd}
A \ar[d,hook,"\ano"] \ar[r,"\in \Wcal_L"] & A' \ar[d,"r"] \\
B \ar[r,"\in \Wcal_L"]   & B' \\
\end{tikzcd}
\]

Then we factor $r$ as an anodyne cofibration followed by a fibration. As the anodyne cofibration (usually only in $\Wcal_R$) is between fibrant objects it is in $\Wcal_L$, hence one obtains a solid diagram:

\[
\begin{tikzcd}
A \ar[d,hook,"\ano"] \ar[r,"\in \Wcal_L"] & A'' \ar[d,->>] \\
B \ar[r,"\in \Wcal_L"] \ar[ur,dotted]   & B' \\
\end{tikzcd}
\]

and a dotted filling. By $2$-out-of-$6$ for $\Wcal_L$ this implies that all maps involved are in $\Wcal_L$ and hence that our anodyne cofibration is indeed in $\Wcal_L$.

\end{proof}

\begin{example}\label{ex:non_Quillen_Loc} We will see in \cref{sec:Bousfield} that if $\Ccal$ is a combinatorial Quillen (or even two-sided) model category then its saturated left and right Bousfield localization $\Lb_S \Ccal$ and $\R_S \Ccal$ are always two-sided model categories. The usual theorem only asserts that these are Quillen model categories if $\Ccal$ is left (resp. right) proper. So any example of a combinatorial Quillen model category whose Bousfield localization is not a Quillen model category (so, does not exist in the usual sense) will provide an example of a two-sided semi-model category that is not a Quillen model category, and hence  do not satisfy the equivalent condition of \cref{prop:QuillenCharac:loc_are_equiv}. 

We have learned from Reid Barton on Mathoverflow \cite{325390} a very nice such example that really allows one to see explicitly these two classes of left and right equivalences: let $\Ccal$ be the four object lattice $d \leqslant c,b \leqslant a$, seen as a category:

\[
\begin{tikzcd}
  a \ar[r] \ar[d] & b \ar[d] \\
c \ar[r] & d \\
\end{tikzcd}
\]

One first considers the Quillen model structure on $\Ccal$, where $a \rightarrow b$ is the only weak equivalence (other than the identities), all maps are fibrations and all maps except $a \rightarrow b$ are cofibrations. One checks that it is a Quillen model structure by verifying by hand that there are no non-trivial lifting problems to solve and that every map has the two required factorizations. However, it fails to be left proper as the square is a pushout along a cofibration, but $c \rightarrow d$ is not an equivalence. Its homotopy category is $a \rightarrow c \rightarrow d$ (indeed $b$ is the only object that is only fibrant and not bifibrant).

One then takes $\Lb_{a \rightarrow c} \Ccal$, the saturated left Bousfield localization of $\Ccal$ at $a \rightarrow c$ (see \cref{sec:Bousfield}). The fibrant objects are the ones which have the lifting property against the cofibration $a \rightarrow c$, i.e. $c$ and $d$, and the corresponding localization of $Ho(\Ccal)$ is equivalent to the category $c \rightarrow d$. The map $a \rightarrow c$ is an anodyne cofibration between cofibrant objects, so it is an equivalence in both the left and the right semi-model structure $\Lb_{a \rightarrow c} \Ccal$, hence $a$ is sent to $c$ in the homotopy category, but the status of $b$ is more subtle: the map $a \rightarrow b $ is an anodyne fibration in $\Ccal$, so it is still an anodyne fibration in $\Lb_{a \rightarrow c} \Ccal$, which makes it a left equivalence (but as $b$ is not fibrant, it is not automatically a right equivalence). On the other hand, the map $b \rightarrow d$ is an anodyne cofibration in $\Lb_{a \rightarrow c} \Ccal$ (it is a pushout of $a \rightarrow c$) so it is a right equivalence (but as $b$ is not cofibrant, not automatically a left equivalence). 

Also note that the composite of these two maps $a \rightarrow b \rightarrow d$ is not an equivalence. So in particular none of the two maps we considered above is in both classes (otherwise their composite would be an equivalence). Here, the left localization functor sends $b$ to $c$, while the right localization functor sends $b$ to $d$.

\end{example}

\begin{remark} In \cite{barton2020model}, R.~Barton has introduced a notion of ``Relaxed premodel category'' which is very similar to our definition. Barton's notion is defined in terms of existence of sufficient simplicial resoluion of cofibrant objects in the arrow category of both $\Ccal$ and $\Ccal^^op$.  It is easy to see, using the usual construction of such resolution in a Quillen model category, that every two-sided model category admit such resolution and hence is a relaxed premodel category. The converse does not quite holds as Barton's definition do not include saturation conditions, but it is also clear that any saturated relaxed premodel category is a two-sided model category as the strong cylinders and strong path objects can be obtained from the first level of the simplicial and cosimplicial resolution. 

However relaxed premodel categories do not corresponds to a completely ``unsaturated'' version of \cref{def:twosided} where one simply drop the saturation condition (the third point). Indeed the resolution required in the definition are formulated in terms of anodyne map, instead of acyclic, so for example every cofibrant object has a strong cylinder object whose legs are anodyne. It seems the reason Barton makes this choice is mostly for simplicity: it is easier to formulate and this automatically holds in the case of an enriched premodel category.
\end{remark}

\section{Generalized Cisinski-Olschok's theory}
\label{sec:Cisinki_Olschok}

In \cite{cisinski2002theories} and \cite{cisinski2006prefaisceaux}, D.-C.~Cisinski has shown how to construct in a systematic way Quillen model structures on toposes whose cofibrations are the monomorphisms. In \cite{olschok2011left}, M.~Olschok has given a partial generalization of this to locally presentable categories that are not necessarily toposes and to an arbitrary cofibrantly generated classes of cofibrations, under the assumption that there exists well behaved cylinder objects and that every object is cofibrant (both of the assumptions being automatically satisfied in the special case treated by Cisinski).

In \cite{henry2018weakmodel} we gave a version of Cisinski-Olschok theory for weak model categories, which we will recall (and rephrase) as \cref{th:weak_Quil_cylinder}. The main goal of this section is to show how this, combined with \cref{Th:recog_LMS,Th:recog_RMS}, allows one to recover Olschok's theorem and many generalizations of it. One can obtain various versions of the theorem depending on the saturation assumption one requires, or whether one uses \cref{Th:recog_LMS} or \cref{Th:recog_RMS}. The original version of Olschok's theorem corresponds to \cref{cor:Olschoktheorem}.

\begin{construction}\label{constr:Pushout_product_property} Let $\Ccal$ and $\Dcal$ be two premodel categories. Let $F,G\colon  \Ccal \rightrightarrows \Dcal$ two left adjoint functors and $\lambda\colon  F \rightarrow G$ be a natural transformation. 

For any arrow $v\colon  X \rightarrow Y$ in $\Ccal$ one denotes by $\lambda \corner{\otimes} v$ the corner-product arrow:

\[ F(Y) \coprod_{F(X)} G(X) \rightarrow G(Y) \]

The reader unfamiliar with the notion can consult the appendix of \cite{joyal2006quasi} for its basic properties.
\end{construction}

\begin{definition}\label{def:(triv)cof_homCat} let $\lambda\colon F \rightarrow G$ be a natural transformation between two left adjoint functors between premodel categories as above.

\begin{itemize}
\item $\lambda$ is said to be a cofibration if for all cofibration $i$, $\lambda \corner{\otimes} i$ is a cofibration and for all anodyne cofibration $j$, $\lambda \corner{\otimes} j$ is an anodyne cofibration.

\item  $\lambda$ is said to be an anodyne cofibration if for all cofibration $i$, $\lambda \corner{\otimes} i$ is an anodyne cofibrations.

\end{itemize}
\end{definition}

We remind the reader that to test whether $\lambda$ satisfies these properties it is enough to check it on the generating cofibrations and generating anodyne cofibrations.

\begin{example}
A left adjoint functor $F\colon  \Ccal \rightarrow \Dcal$ between premodel categories is ``cofibrant'' in the sense of \cref{def:(triv)cof_homCat} if and only if $F$ is left Quillen functor.
\end{example}

\begin{remark}
It can be shown that this notion of cofibrations and anodyne cofibration on the category $[\Ccal,\Dcal]$ of left adjoint functor between two combinatorial categories $\Ccal$ and $\Dcal$ form a combinatorial premodel structure on the category $[\Ccal,\Dcal]$. Moreover this combinatorial premodel structure on $[\Ccal,\Dcal]$ is the exponential object of a symmetric monoidal closed structure on the category of combinatorial premodel categories (and left Quillen functor between them) such that the map $F\colon  \Ccal \otimes \Dcal \rightarrow \Ecal$ corresponds to Quillen bi-functors. This is used much more seriously in \cite{barton2020model}, and we will also use it heavily in \cite{henry2020model}.
\end{remark}

\begin{remark}\label{rk:adjoint_trans_nat} The category of left adjoint functor from $\Ccal$ to $\Dcal$ is equivalent to the category of left adjoint functors from $\Dcal^\op$ to $\Ccal^\op$. The equivalence sends a functor $F:\Ccal \to \Dcal$ to the opposite of its right adjoint $F_*^\op : \Dcal^\op \to \Ccal^\op $, a natural transformation $\lambda: F \to G $ corresponds to a natural transformation $\lambda_*: G_* \to F_*$, but as we work with functor valued in $\Ccal^{\op}$, this is a natural transformation $\lambda: F_*^{\op} \to G_*^\op$, making the equivalence above covariant. We have:\end{remark}

\begin{lemma}\label{lem:adjoint_of_cofibration} A natural transformation $\lambda: F \to G$ between two left adjoint functors $\Ccal \to \Dcal$ as above is an (anodyne) cofibration in the sense of \cref{def:(triv)cof_homCat} if and only if the corresponding natural transformation $\lambda_*^\op : F_*^\op \to G_*^\op$ is an (anodyne) cofibration.
\end{lemma}

Here we use the opposite premodel structure on $\Dcal^\op$ and $\Ccal^\op$ whose cofibrations are the fibrations of $\Dcal$ and $\Ccal$. Explicitly, the lemma says for example that if $\lambda$ is a cofibration in the sense of \cref{def:(triv)cof_homCat}, then for each fibration $f:X \to Y$ in in $\Ccal$ the map:

\[ G_*(X) \to G_*(Y) \times_{F_*(Y)} F_*(X) \]

is a fibration, and an anodyne fibration when $f$ is an anodyne fibration.

\begin{definition}\label{def:WeakQuillenCylinder}

A \emph{weak Quillen cylinder} on a premodel category $\Ccal$ is a pair of left adjoint functors $I,D \colon  \Ccal \rightarrow \Ccal$ endowed with a diagram of natural transformations:

\[
\begin{tikzcd}
  \id_{\Ccal} \coprod \id_{\Ccal} \ar[r,hook,"i"] \ar[d,"\nabla"] & I \ar[d,"e"] \\
 \id_{\Ccal} \ar[r,hook,"\ano", "j"{swap}] & D \\ 
\end{tikzcd}
\]

where $\nabla$ denotes the codiagonal map and $i$ is a cofibration, $j$ is an anodyne cofibration and the composite of $i$ with the first coproduct inclusion $\id_{\Ccal} \rightarrow I$ is an anodyne cofibration (all in the sense of \cref{def:(triv)cof_homCat}).

\end{definition}

So, roughly, it is just a weak cylinder object for the identity functor in the category of endofunctors (except $i$ is required to be an anodyne cofibration and not just an acyclic cofibration). 

\begin{remark}\label{rk:C^op_has_a_Qcyl} By \cref{rk:adjoint_trans_nat,lem:adjoint_of_cofibration}, if $I,D:\Ccal \to \Ccal$ form a weak Quillen cylinder on $\Ccal$ then $I_*^\op$ and $D_*^\op$ form a weak Quillen cylinder on $\Ccal^\op$.
\end{remark}

\begin{lemma}\label{lem:pushout_comp_are_cof} In a premodel category, given a diagram:

\[ \begin{tikzcd}
 B \ar[d] & \ar[l] A \ar[d] \ar[r] & C \ar[d] \\
 B' & \ar[l] A' \ar[r] & C' \\
\end{tikzcd} \]

If the maps $C \to C'$ and $B \coprod_A A' \to B'$ are cofibrations (resp. anodyne cofibrations) then the comparison map:

\[ B \coprod_A C \to B' \coprod_{A'} C' \]
  
is a cofibration (resp. anodyne cofibration).
\end{lemma}

\begin{proof} This map is a composite:

\[ B \coprod_A C \to B \coprod_A C' \to  B' \coprod_{A'} C' \]
We claim that the first map is a pushout of $C \to C'$ and the second is a pushout of $B \coprod_A A' \to B'$. Indeed consider the two diagrams:

\[
\begin{tikzcd}
  A \ar[d] \ar[r] & C \ar[d] \ar[r] & C' \ar[d] &  A' \ar[d] \ar[r] & \displaystyle B \coprod_A A' \ar[r] \ar[d] & B' \ar[d] \\ 
B \ar[r] & \displaystyle B \coprod_A C \ar[r] & B \displaystyle \coprod_A C'&  C' \ar[r] &\displaystyle B \coprod_A C' \ar[r] &\displaystyle B' \coprod_{A'} C' \\
\end{tikzcd}
\]

In both cases the square on the left and the outside rectangle are pushouts.  Hence it follows that the right square is also a pushout. For the left square of the second diagram, it is a pushout by the same argument as for the map $B \coprod_A C \to B \coprod_A C'$. \end{proof}

The results of Section $3$ of \cite{henry2018weakmodel} can be rephrased (in slightly less generality) as:

\begin{theorem}\label{th:weak_Quil_cylinder}
If a premodel category admits a weak Quillen cylinder, then it is a weak model category.
\end{theorem}

This was proved in \cite{henry2018weakmodel}, though the language of the present paper, and the simplified version of the statement allows us to give a simplified version of the proof:

\begin{proof} 
Let $\Ccal$ be a premodel category with a weak Quillen cylinder $(I,D,i,j,e)$ as in \cref{def:WeakQuillenCylinder}. Let $f\colon  A \hookrightarrow B$ be any core cofibration, we form the diagram:

\[\begin{tikzcd}
 \displaystyle  B \coprod_A B \ar[r,hook] \ar[d,"\nabla"] & \displaystyle IB \coprod_{IA} DA \ar[d] \\
B \ar[r,hook,"\ano"] & DB \\ 
\end{tikzcd} \]

We claim that it is a ``relative weak cylinder object'' for $f$. Indeed the map $B \overset{\ano}{\hookrightarrow} DB$ is $j \corner{\otimes} (\emptyset \to B)$ and hence an anodyne cofibration. The map $B \coprod_A B \hookrightarrow  IB \coprod_{IA} D$ is the comparison map between the pushouts of the rows of:

\[ \begin{tikzcd}
  B \coprod B \ar[d,hook] & \ar[l,hook] A \coprod A \ar[d,hook] \ar[r] & A \ar[d,hook] \\
IB & IA \ar[l,hook] \ar[r] & DA \\
\end{tikzcd} \]

Hence it is a cofibration by \cref{lem:pushout_comp_are_cof} as $A \hookrightarrow DA$ is an (anodyne) cofibration and the map $(B \coprod B) \coprod_{A \coprod A} IA \hookrightarrow IB$ is $i \corner{\otimes} (A \hookrightarrow B)$ and hence is a cofibration.

Similarly, the map $B \hookrightarrow  IB \coprod_{IA} D$ is an anodyne cofibration by \cref{lem:pushout_comp_are_cof} applied to:

\[ \begin{tikzcd}
  B \ar[d,hook] & \ar[l,hook] A \ar[d,hook] \ar[r] & A \ar[d,hook] \\
IB & IA \ar[l,hook] \ar[r] & DA \\
\end{tikzcd} \]

Dually, by \cref{rk:C^op_has_a_Qcyl}, the premodel category $\Ccal^\op$ also admits a weak Quillen cylinder, hence it also admits relative weak cylinder objects, which means that $\Ccal$ has relative weak path objects, and hence that $\Ccal$ is a weak model category.

\end{proof}

Cisinski-Olschok theory is concerned instead with combinatorial premodel categories endowed with what we will call a strong Quillen cylinder, i.e. a weak Quillen functor where the map $j\colon \id_{\Ccal} \rightarrow D$ is an isomorphism. More explicitly:

\begin{definition}\label{def:Strong_Quillen_cylinder}
 A \emph{strong Quillen cylinder} on a premodel category $\Ccal$ is a left adjoint endofunctor $I \colon  \Ccal \rightarrow \Ccal$ endowed with natural transformations:

\[ \id_{\Ccal} \coprod \id_{\Ccal} \overset{i}{\hookrightarrow} I \rightarrow \id_{\Ccal} \]

factoring the codiagonal map such that $i$ is a cofibration and the composite $\id_{\Ccal} \rightarrow I$ of $i$ with the first coproduct inclusion is an anodyne cofibration, both in the sense of \cref{def:(triv)cof_homCat}.
\end{definition}

\begin{theorem}\label{th:strong_cylinder_Olschok} Let $\Ccal$ be a premodel category equipped with a strong Quillen cylinder. Then:

\begin{itemize}

\item Any core left saturated and right saturated premodel structure on the underlying category of $\Ccal$ that has the same left core as $\Ccal$ is a left semi-model category.

\item Any core right saturated and left saturated premodel structure on the underlying category of $\Ccal$ that has the same right core as $\Ccal$ is a right semi-model category.

\end{itemize}
\end{theorem}

In particular, under the same assumptions, $\Lb \R \Ccal$ and $\R \Lb \Ccal$ are both left and right semi-model categories, i.e. they are two-sided model category in the sense of \cref{sec:Two_sided}. If one use Fresse's definition instead of Spitzweck's definition of left semi-model structure, than one can drop right saturation in the first point, and left saturation in the second point.

\begin{proof} \Cref{th:weak_Quil_cylinder} shows that $\Ccal$ is a weak model category. Moreover, if $I$ is the cylinder functor, then for any cofibrant object $X$, $IX$ is a strong cylinder object for $X$, and for $X$ fibrant and $P$ the right adjoint of $I$, $PX$ is a strong path object for $X$. Hence $\Ccal$ satisfies all the core part of the assumption of \cref{Th:recog_LMS,Th:recog_RMS}. Any category with the same right/left core as $\Ccal$ also satisfies these assumptions, hence any category with the same right/left core satisfying the appropriate saturation assumption is going to be a left or right semi-model category.
\end{proof}

\begin{notation}\label{nota:Connection_with_Olschok} The definitions and results in this section are formulated in a slightly different way to those in Cisinski's and Olschok's work. We introduce a few other concepts in order to clarify the connection:

\begin{itemize}

\item A \emph{structured category} $\Ccal$ is a complete and cocomplete category endowed with a weak factorization system (cofibration,anodyne fibration). It is said to be accessible (resp. combinatorial) if it is a locally presentable category and its weak factorization system is accessible (resp. cofibrantly generated).

\item Given a natural transformation $\lambda\colon F \rightarrow G$ between two left adjoint functors $F,G \colon  \Ccal \rightrightarrows \Dcal$ between two structured categories, one says that $\lambda$ is a \emph{cofibration} if $\lambda \corner{\otimes} i$ is a cofibration in $\Dcal$ for all cofibrations $i$ of $\Ccal$.

\item A \emph{strong pre-Quillen cylinder} on a structured category $\Ccal$ is a left adjoint endofunctor $C \colon  \Ccal \rightarrow \Ccal$ endowed with a natural transformation:

\[ \id_{\Ccal} \coprod \id_{\Ccal} \overset{\sigma}{\hookrightarrow} C \rightarrow \id_{\Ccal} \]

factoring the codiagonal map and such that $\sigma$ is a cofibration.

\end{itemize}
\end{notation}

\begin{construction}\label{constr:Oslchok_Lambda} Definition 3.5 of \cite{olschok2011left} can be reformulated as follows. Let $\Ccal$ be a combinatorial structured category equipped with a strong pre-Quillen cylinder:

\[ \id_{\Ccal} \coprod \id_\Ccal \overset{\sigma=(\sigma_0,\sigma_1)}{\hookrightarrow} C \rightarrow \id_\Ccal .\]

Moreover, let $I$ be a set of generating cofibrations in $\Ccal$, and $S$ any set of cofibration in $\Ccal$. We define $\Lambda(C,S,I)$ to be the smallest set of cofibrations in $\Ccal$ such that:

\begin{itemize}

\item $S \subset \Lambda(C,S,I)$.

\item $\forall i \in I$, $\sigma_0 \corner{\otimes} i \in \Lambda(C,S,I)$ and $\sigma_1 \corner{\otimes} i \in\footnote{This second condition is only here so that our definition match Definition 3.5 of \cite{olschok2011left}, otherwise it is enough to have it for $\sigma_0$ to obtain \cref{cor:Olschoktheorem}.} \Lambda(C,S,I)$.

\item if $j \in \Lambda(C,S,I)$ then $\sigma \corner{\otimes} j \in \Lambda(C,S,I)$.

\end{itemize}

Then $\Lambda(C,S,I)$ generates the (anodyne cofibration, fibration) weak factorization system of a premodel structure on $\Ccal$ which makes $C$ into a strong Quillen cylinder. In fact it generates the smallest class of anodyne cofibrations containing $S$ making this true.
\end{construction}

The connection with Olschok's terminology from \cite{olschok2011left} (closely connected to Cisinski's terminology in \cite{cisinski2006prefaisceaux}) is as follows: In \cite{olschok2011left} one restricts to combinatorial structured categories in which every object is cofibrant, what we have called a pre-Quillen cylinder is exactly what is called a cartesian cylinder in \cite{olschok2011left}. This terminology ``cartesian'' of \cite{olschok2011left} comes from \cite{cisinski2006prefaisceaux} where the cylinder considered are obtained using cartesian product by a fixed object in a topos, and whose key properties can be expressed in terms of certain cartesian square, which justify this terminology ``cartesian''. We prefer to avoid using this terminology in the more general situation where no cartesian structure plays a role. Finally the set $\Lambda(C,S,I)$ defined above is exactly as from Definition $3.5$ in \cite{olschok2011left}.

And one can indeed recover as a corollary of the results of this section, M.~Olschok's Theorem 3.16 of \cite{olschok2011left}.

\begin{cor}[Cisinki-Olschok theorem]\label{cor:Olschoktheorem} Let $\Ccal$ be a combinatorial structured category (see \cref{nota:Connection_with_Olschok}) in which every object is cofibrant. And let $C$ be a strong pre-Quillen cylinder on $\Ccal$ and $S$ any set of cofibration of $\Ccal$. Then $\Ccal$ admits a Quillen model structure such that:

\begin{itemize}
\item Its cofibrations are the cofibrations of $\Ccal$.

\item Its fibrant objects and fibrations between fibrant objects are characterized by the lifting property against the set $\Lambda(C,S,I)$ from \cref{constr:Oslchok_Lambda}, where $I$ is any generating set of cofibrations.

\item Equivalences are maps $f\colon X \rightarrow Y$ such that for all fibrant object $T$, $f$ induced a bijection on the set of homotopy classes $[Y,T] \rightarrow [X,T]$ where the homotopy relation is defined as maps $CX \rightarrow T$.

\end{itemize}

\end{cor}
\begin{proof} $\Ccal$ with $I$ and $\Lambda(C,S,I)$ as generating cofibrations and anodyne cofibrations is a combinatorial category with a strong Quillen cylinder given by $C$. As every object of $\Ccal$ is cofibrant, it is right saturated and its left saturation and core left saturation coincide. Hence by \cref{th:strong_cylinder_Olschok} its (core) left saturation is a left semi-model category, and as every object is cofibrant it is actually a Quillen model category. The first two points follows immediately from the construction. The last point is just a consequence of the fact that as every object is cofibrant, then for any fibrant object $T$, the homotopy class of maps from $X$ to $T$ defined using the cylinder $CX$ are indeed the morphisms in the homotopy category, and hence one can test whether a map is an equivalence by using these.
\end{proof}

\begin{remark} Given $\Ccal$ an accessible structured category equipped with a pre-Quillen cylinder, it is possible (though we will not do it here), to construct a smallest premodel structure on $\Ccal$ making it into a Quillen cylinder (as it is achieved by \cref{constr:Oslchok_Lambda} in the combinatorial case). We will not detail this further, but using such a construction (instead of \cref{constr:Oslchok_Lambda}) one can also give a version of \cref{cor:Olschoktheorem} for an accessible structured category.\end{remark}

\begin{remark} We also obtain a dual form of Olschok's theorem: if $\Ccal$ is a combinatorial or accessible premodel category with a strong Quillen cylinder and where every object of $\Ccal$ is fibrant then it admits a Quillen model structure with the same fibrations and the same core (acyclic) cofibrations. Indeed, as every object is fibrant, $\Ccal$ is left saturated so that $\R \Ccal$ is bi-saturated and has the same fibrations and core cofibrations as $\Ccal$. Therefore by \cref{th:strong_cylinder_Olschok} it is a right semi-model structure and hence, as every object, is fibrant it is a full Quillen model structure.
\end{remark}

\begin{example}
An example where this dual Cisinki-Olschok's theorem would have been useful is the construction of the ``Folk'' model structure on the category of strict $\infty$-category in \cite{lafont2010folk}. They explicitly construct a path object functor $X \mapsto \Gamma(X)$ in section $4.4$ which can be easily checked to be a right adjoint functor (for example it is $\omega_1$-accessible and preserves all limits),  Theorem 2 of \cite{lafont2010folk} equips it with all the natural transformations to make it a ``path object functor'' and subsequent results, e.g. Corollary $2$, in \cite{lafont2010folk} shows that it is the right adjoint of a strong Quillen cylinder. As every object is fibrant in this model structure this directly concludes the proof that the folk model structure exists and is a Quillen model structure.
\end{example}

\begin{remark}\label{rk:simplicial} Another striking consequences of these results is that any premodel category that is enriched over a monoidal model category is immediately a weak model category, and if it satisfies some of the saturation condition of \cref{def:saturation}, it will be a left/right/two-sided model category. For example, any bi-saturated simplicial premodel category is automatically a two-sided model category.  Of course, by ``enriched over a monoidal model category'' or ``simplicial premodel category'' one means that the usual compatibility conditions between the enrichment and weak factorization system familiar to the theory of enriched (or simplicial) model categories are satisfied.

A further corollary of this is that the category of left (resp. right, resp. two-sided) simplicial combinatorial (or even accessible) model categories has all limits and colimits. Indeed limits and colimits of combinatorial and accessible premodel categories exists and are discussed in \cref{sec:apendix_awfs}, and this easily caries over to completeness and cocompleteness of the category of simplicial accessible (or combinatorial) premodel categories. As (core) left and right saturated categories form  reflective and coreflective subcategories in a compatible way (by \cref{th:main_saturation}) it immediately follows that the full subcategories of categories satisfying some given saturation condition automatically have limits and colimits, hence the result follows from the observation above.
\end{remark}

\section{Left and right Bousfield localizations}
\label{sec:Bousfield}

Before talking about Bousfield localization, we introduce a notation and a lemma that will be useful for both left and right Bousfield localizations:

\begin{notation} Given $i\colon A \hookrightarrow B$ a core cofibration in a weak model category $\Ccal$, we denote by $\nabla i$ the core cofibration:

\[ B \coprod_A B \overset{\nabla i}{\hookrightarrow} I_A B\]

for some choice of a relative weak cylinder object $I_A B$ for $i$. $\nabla i$ is essentially a core cofibration representing the homotopy codiagonal map of $i$. As $\nabla i$ is itself a core cofibration one can iterate the construction to obtain $\nabla^k i$ the ``homotopy higher codiagonal'' of $i$.
\end{notation}

\begin{lemma}\label{lem:nabla_lift} Let $p$ and $q$ be two composable core fibrations in a weak model category $\Ccal$ and let $i$ be a core cofibration. Assume that $q \circ p$ has the right lifting property against $i$ and $q$ has the right lifting property against $\nabla i$, then $p$ has the right lifting property against $i$.
\end{lemma}

\begin{proof}

Consider a lifting problem:

\[
\begin{tikzcd}
  A \ar[d,hook,"i"] \ar[r] & X \ar[d,"p",->>] \\
B \ar[r] \ar[ur,dotted,"?"] & Y \ar[d,"q",->>] \\
 & Z
\end{tikzcd}
\]

using the lifting property of $i$ against $q \circ p$, we obtain a map $?_0 \colon  B \rightarrow X$ making the upper triangle commutes, but where the lower triangle only commutes when post composing it with $q$. We hence have two different maps $B \rightrightarrows Y$ that are equalized by $A \rightarrow B$ and co-equalized by $Y \rightarrow Z$. We see these as a map $ \tau\colon B \coprod_A B \rightarrow Y$ whose composite with $q$ factors through the codiagonal map $B \coprod_A B \rightarrow B$, and as $Z$ is fibrant there is a self-homotopy $I_A B \rightarrow Z$ extending $q \tau$ which gives us the solid diagram on the left below:

\[
\begin{tikzcd}
  B \coprod_A B \ar[r,"\tau"] \ar[d,"\nabla i"swap] & Y \ar[d,"q"] \\
I_A B \ar[r] \ar[ur,dotted,"h"description] & Z \\
\end{tikzcd}
\qquad
\begin{tikzcd}
A \arrow[rd, hook] \arrow[ddd, hook] \arrow[rr] &                                       & X \arrow[ddd, "p", two heads] \\
                                                & B \arrow[d, "\sim", hook] \arrow[ru,"?_0"description] &                               \\
                                                & I_A B \arrow[rd,"h"] \arrow[ruu, dotted]  &                               \\
B \arrow[ru, hook] \arrow[rr]                   &                                       & Y                            
\end{tikzcd}\]

As we assumed $\nabla i$ has the lifting property, one obtains a dotted lift $h$, which then fits into the solid diagram on the right above. It admits a dotted arrow as drawn because $p$ is a core fibration and $B \overset{\sim}{\hookrightarrow} I_A B$ an acyclic cofibration. The composite diagonal arrow $B \rightarrow X$ provides the diagonal filling we needed to conclude the proof.\end{proof}

We start with left Bousfield localizations which given the machinery we have developed in the rest of the paper, are fairly easy:

\begin{theorem}\label{Th:LeftBousfieldMain} Let $\Ccal$ be an accessible (resp. combinatorial) weak model category, and let $S$ be a set of morphisms in $Ho(\Ccal)$. Then there is an accessible (resp. combinatorial) weak model structure $\Lb^c_S \Ccal$ on $\Ccal$ such that:

\begin{enumerate}

\item\label{Th:LeftBousfieldMain:Id_Left_quillen} The identity $\Ccal \rightarrow \Lb^c_S \Ccal$ is a left Quillen functor.

\item\label{Th:LeftBousfieldMain:it_inverts_arrows} The induced functor $Ho(\Ccal) \rightarrow Ho( \Lb^c_S \Ccal)$ sends all elements of $S$ to isomorphisms.

\item\label{Th:LeftBousfieldMain:is_core_saturated} $\Lb^c_S \Ccal$ is core left saturated.

\item\label{Th:LeftBousfieldMain:is_universal} $\Lb^c_S \Ccal$ is universal for the three properties above.

\end{enumerate}

Moreover:

\begin{enumerate}[resume]

\item\label{Th:LeftBousfieldMain:same_cof} $\Lb^c_S \Ccal$ has the same cofibrations as $\Ccal$.

\item\label{Th:LeftBousfieldMain:kinda_same_fib} A map between fibrant objects of $\Lb^c_S \Ccal$ is a fibration of $\Lb^c_S \Ccal$ if and only if it is a fibration of $\Ccal$.

\end{enumerate}

\end{theorem}

\begin{proof}

For each $s \in S$, we choose a representative $j$ of $s$ which is a core cofibration, and we also makes choices of $\nabla^k j$ for all $k \geqslant 0$ and $j$. Let $J$ be the set of all these cofibrations. 

Consider the accessible (resp. combinatorial) premodel structure $\Lcal_S \Ccal$ on $\Ccal$ which has the same cofibrations as $\Ccal$ and whose anodyne cofibrations are generated by the anodyne cofibrations of $\Ccal$ and all the maps in $J$, that is the supremum in the sense of \ref{nota:before_main_th} of the accessible weak factorization system (anodyne cofibration, fibrations) and the cofibrantly generated weak factorization system generated by $J$ (whose existence is guarantied by \cref{th:appendix_Main}.\ref{th:appendix_Main:sup_wfs}).

Note that $\Lcal_S \Ccal$ depends on the choice of the representative $j$ and of the $\nabla^k j$ we made, but we will show that its core left saturation $\Lb^c_S \Ccal \colon = \Lb^c \Lcal_S \Ccal$ has the properties claimed by the theorem and hence only depends on $\Ccal$ and $S$ because of the universal property (point \ref{Th:LeftBousfieldMain:is_universal}).

Properties \ref{Th:LeftBousfieldMain:is_core_saturated} and \ref{Th:LeftBousfieldMain:same_cof} are immediate by construction (core left saturation do not affect the cofibrations). We then prove that $\Lcal_S \Ccal$ satisfies property \ref{Th:LeftBousfieldMain:kinda_same_fib}. If $f\colon X \rightarrow Y$ is a $\Ccal$-fibration between $\Lcal_S \Ccal$-fibrant objects then, for any $j \in J$, as $X$ has the lifting property against $j$ and $Y$ has the lifting property against $\nabla j \in J$, \cref{lem:nabla_lift} (applied with $q\colon Y \rightarrow 1$) shows that $f$ has the lifting property against $j$. Because core right saturation do not affect the core, this shows that $\Lb^c_S \Ccal$ satisfies \ref{Th:LeftBousfieldMain:kinda_same_fib} as well.

We then show that $\Lcal_S \Ccal$ is indeed a weak model category. It has the same cofibrations and more anodyne cofibrations than $\Ccal$, so every cofibration clearly has a relative weak cylinder object (because it has one in $\Ccal$), and similarly, any fibration between fibrant objects in $\Lcal_S \Ccal$ has a relative weak path object because it has one in $\Ccal$, which is still one in $\Lcal_S \Ccal$ due to property \ref{Th:LeftBousfieldMain:kinda_same_fib} for $\Lcal_S \Ccal$ proved above.

Property \ref{Th:LeftBousfieldMain:it_inverts_arrows} is clear as well: every arrow in $S$ has a representative which is a core cofibration which becomes an anodyne cofibration in $\Lcal_S \Ccal$, so it is sent to an isomorphism in $Ho(\Lcal_S \Ccal)$. It only remains to check property \ref{Th:LeftBousfieldMain:is_universal}, i.e. the universality:

Let $L\colon \Ccal \rightarrow \Dcal$ be a left Quillen functor satisfying \ref{Th:LeftBousfieldMain:it_inverts_arrows} with $\Dcal$ core left saturated.  We have to show that any arrow in $J$ is sent to an anodyne cofibration in $\Dcal$: this will immediately imply that $L$ factor (uniquely) through $\Lcal_S \Ccal$, and hence as $\Dcal$ is core left saturated through $\Lb^c_S \Ccal$. Now this claim is easily proved by induction: if $j \in J$ is a core cofibration representing an arrow in $S$, then its image by $L$ is a core cofibration and an equivalence, hence is an anodyne cofibration by core left saturation of $\Dcal$, and if $j = \nabla j'$ for a core cofibration $j'$ such that $L j'$ is an anodyne cofibration in $\Dcal$, then $L j$ is a $\nabla L j'$ hence is an equivalence (as $j'$ is an anodyne cofibration) and hence is also an anodyne cofibration by core saturation of $\Dcal$.
\end{proof}

\begin{remark}\label{Rk:LeftSat_Bousfield} Instead of the core left saturation, one can take the left saturation of $\Lcal_S \Ccal$:

\[ \Lb_S \Ccal \coloneqq \Lb \Lcal_S \Ccal = \Lb \Lb_S^c \Ccal \]

which also satisfies all the condition of theorem, except that core left saturation is replaced by left saturation.

\end{remark}

\begin{remark}\label{rk:LBloc_at_classes} In the case of combinatorial model categories, \cref{Th:LeftBousfieldMain} seems optimal in the sense that if one wants the localization to be combinatorial it has to be a localization at a set of morphisms (because there will be in the localization a generating set of anodyne cofibrations). But in the accessible case this is no longer the case: if we only want the localization to be accessible it might happen that there are situations where a left Bousfield localization exists without it being a left Bousfield localization at a \emph{set} of maps.

Given a \emph{class} $S$ of arrows in $\Ccal$, one says that the left Bousfield localization of $\Ccal$ at $S$ exists if for some choices of a set $J_0$ of cofibrations representing all the elements of $S$, and of $J$ a closure of this set of representatives under $\nabla$, such the weak factorization system generated by $J$ exists and is accessible. If this happen the rest of the proof of \cref{Th:LeftBousfieldMain} carries over and one can construct the Bousfield localization $\Lb_S^c  \Ccal$ (or $\Lb_S  \Ccal$) and it has all the properties claimed in \cref{Th:LeftBousfieldMain}. The following proposition gives a characterization of classes for which this work, which has the advantages of being a necessary and sufficient condition, and of making left localization and right localization look extremely similar, but does not seems to very be convenient in practice. I have not been able to find a better characterization yet, but I still believe it exists.
\end{remark}

\begin{prop}\label{prop:Bloc_at_class_inverted} Consider a Quillen adjunction between two accessible weak model categories:
\[ L \colon  \Ccal \leftrightarrows \Dcal \colon  R \]

and let $S$ be the class of arrows in $Ho(\Ccal)$ which are inverted by $L \colon  Ho(\Ccal) \rightarrow Ho(\Dcal)$. Then the left Bousfield localization of $\Ccal$ at $S$ exists (in the sense of \cref{rk:LBloc_at_classes}) and is accessible. If $\Ccal$ and $\Dcal$ are combinatorial, this localization is also combinatorial.
\end{prop}

As the proof below will show, if $\kappa$ is an \emph{uncountable} regular cardinal, $\Ccal$ and $\Dcal$ are $\kappa$-accessible (resp. $\kappa$-combinatorial) weak model categories and $(L,R)$ is a $\kappa$-adjunction in the sense of \cref{not:kappa_adjunction}, then the localizations $\Lb^c_S \Ccal$ and $\Lb_S \Ccal$ are also $\kappa$-accessible (resp. $\kappa$-combinatorial).

\begin{proof}

We take $J$ to be the class of all core cofibrations whose image by $L$ is an acyclic cofibrations in $\Dcal$. This is a set of representatives of $S$ and it is easy to see that it is already closed under $\nabla$ and $\nabla^k$. Moreover the accessible weak factorization system generated by $J$ exists, and is appropriately $\kappa$-accessible or $\kappa$-combinatorial, because:

\begin{itemize}

\item The class of maps whose image by $L$ is an acyclic cofibration in $\Dcal$ is the left class of a (left transferred) weak factorization system on $\Ccal$ by \cref{th:appendix_Main}.\ref{th:appendix_Main:left_transfer}.

\item The class of cofibrations whose image by $L$ are acyclic cofibrations in $\Dcal$ is also the left class of weak factorization system: it is the infimum in the sense of \cref{nota:before_main_th} of the accessible weak factorization constructed in the previous point, and the (cofibration,anodyne fibration) weak factorization system. Hence it exists and is accessible by \cref{th:appendix_Main}.\ref{th:appendix_Main:inf_wfs}

\item Finally, the class of maps with the right lifting property against the core cofibration sent to acyclic cofibrations (i.e. against $J$ above) is the right class of an accessible weak factorization system by \cref{lem:Constructing_some_WFS}.\ref{lem:Constructing_some_WFS:1} applied to the previous weak factorization system and the (cofibration,anodyne fibration) weak factorization system.

\end{itemize}

\end{proof}

\begin{remark}\label{Rk:SlocalObjects} A typical property of Bousfield localization is missing in our statement of \cref{Th:LeftBousfieldMain}: that the fibrant objects of $\Lb^c_S \Ccal$ are exactly the object that are fibrant in $\Ccal$ and $S$-local. The reason why we did not include this is that as we have not defined a simplicial hom space in the framework of weak model categories yet, and we would need this to define what are $S$-local objects. But it should also be clear from the construction of $\Lb^c_S \Ccal$ in the proof of \cref{Th:LeftBousfieldMain} (more precisely of ``$\Lcal_S \Ccal$'') that this condition is satisfied for all reasonable definitions of the simplicial hom spaces: $\Lb_S \Ccal$, $\Lb^c_S \Ccal$ and $\Lcal_S \Ccal$ all have the same fibrant objects, and an object of $\Lcal_S \Ccal$ is fibrant if and only if it is fibrant in $\Ccal$ and has the lifting property against choice of core cofibrations representing all maps in $S$ as well as all their higher homotopy codiagonal, which exactly express $S$-locality.
\end{remark}

\begin{remark}\label{rk:Bousfield_on_Ho_cat}Given an accessible weak model category $\Ccal$, and a left Bousfield localization $\Lb_S \Ccal$ (or $\Lb^c_S \Ccal$). One easily see that the right Quillen functor $Id\colon  \Lb_S \Ccal \rightarrow \Ccal$ induces a fully faithful functor $Ho(\Lb_S \Ccal) \rightarrow Ho(\Ccal)$. Indeed, this follows immediately from the fact that the bifibrant objects of $\Lb_S \Ccal$ are a full subcategory of the bifibrant objects of $\Ccal$, and that the path object or cylinder constructed in $\Ccal$ for objects that are bifibrant in $\Lb_S \Ccal$ are still path objects (or cylinder object) in $\Lb_S \Ccal$, so the homotopy relation is the same in both categories.

As a Quillen adjunction induces an adjunction at the level of homotopy categories (see for example \cite[Prop. 2.4.3]{henry2018weakmodel}), so this shows that $Ho(\Lb_S \Ccal)$ is a reflective full subcategory of $Ho(\Ccal)$.
\end{remark}

We now move to right Bousfield localizations. There is a well-known additional difficulty with right Bousfield localization, which is that in order to get a combinatorial localization one should not take a right localization at a set of maps, but rather something like the localization at a class of maps that are homotopically right orthogonal to some set of objects (or of arrow). We introduce the following definition to handle this:

\begin{definition}\label{def:right_localizer} Let $\Ccal$ be a $\kappa$-accessible weak model category, and $S$ a class of arrows in $Ho(\Ccal)$. One says that $S$ is a \emph{$\kappa$-accessible right localizer} if there exists a $\kappa$-accessible right Quillen functor $R \colon  \Ccal \rightarrow \Dcal$ with $\Dcal$ a $\kappa$-accessible weak model category, such that $S$ is the class of all arrows that are sent to isomorphisms by $Ho(R) \colon  Ho(\Ccal) \rightarrow Ho(\Dcal)$.

One says that $S$ is a \emph{$\kappa$-combinatorial right localizer} if furthermore $\Ccal$ and $\Dcal$ are $\kappa$-combinatorial.

One says that $S$ is an accessible (resp. combinatorial) right localizer if it is a $\kappa$-accessible (resp. $\kappa$-combinatorial) right localizer for some regular cardinal $\kappa$.
\end{definition}

\begin{remark} In the literature, one usually considers right Bousfield localization at the class of ``$K$-colocal equivalences'' for $K$ a set of objects. Similarly to \cref{Rk:SlocalObjects}, as we have not defined yet simplicial hom spaces in the framework of weak model categories, one cannot really use this definition here. But it should be easy to convince any reader already familiar with these notions that our definition of combinatorial right localizer is equivalent to classes of $K$-colocal equivalences for some set $K$ of objects: given a set of objects $K$, one choose a cosimplicial resolution $(k_i)$ of each object $k \in K$. Each such resolution induce a left Quillen functor $\widehat{\Delta} \rightarrow \Ccal$, and together they give rise to a right Quillen functor $\Ccal \rightarrow (\widehat{\Delta})^K$. The maps in $Ho(\Ccal)$ sent to isomorphisms in $Ho(\widehat{\Delta}^K)$ by this functor are, essentially by definition, the $K$-colocal equivalences.

Conversely, the combinatorial right localizer attached to a right Quillen functor $R\colon  \Ccal \rightarrow \Dcal$ is easily seen to be the class of $K$-colocal equivalence for $K$ the set of image by the left adjoint of $R$ of all the $\kappa$-presentable cofibrant objects of $\Dcal$, for any $\kappa$ such that $\Dcal$ is $\kappa$-presentable.
\end{remark}

\begin{theorem}\label{th:RightBousfieldMain} Let $\Ccal$ be an accessible (resp. combinatorial) weak model category, and let $S$ be any accessible right localizer of $\Ccal$. Then there exists an accessible weak model structure $\R^c_S \Ccal$ such that:

\begin{enumerate}

\item\label{Th:RightBousfieldMain:Id_Right_quillen} The identity $\Ccal \rightarrow \R^c_S \Ccal$ is a right Quillen functor.

\item\label{Th:RightBousfieldMain:it_inverts_arrows} The induced functor $Ho(\Ccal) \rightarrow Ho( \R^c_S \Ccal)$ send all elements of $S$ to isomorphisms.

\item\label{Th:RightBousfieldMain:is_core_saturated} $\R^c_S \Ccal$ is core right saturated.

\item\label{Th:RightBousfieldMain:is_universal} $\R^c_S \Ccal$ is universal for the three properties above, and $S$ is the right localizer defined by $\Ccal \rightarrow \R^c_S \Ccal$.

\end{enumerate}

Moreover:

\begin{enumerate}[resume]

\item\label{Th:RightBousfieldMain:same_fib} $\R^c_S \Ccal$ has the same fibrations as $\Ccal$.

\item\label{Th:RightBousfieldMain:kinda_same_cof} A map between cofibrant objects of $\R^c_S \Ccal$ is a cofibration of $\R^c_S \Ccal$ if and only if it is a cofibration of $\Ccal$.

\end{enumerate}

Given an \emph{uncountable} regular cardinal $\kappa$, If $\Ccal$ is $\kappa$-accessible (resp. $\kappa$-combinatorial) and $S$ is a $\kappa$-accessible (resp. $\kappa$-combinatorial) right localizer on $\Ccal$, then $\R^c_S \Ccal$ is $\kappa$-accessible (resp. $\kappa$-combinatorial).

\end{theorem}

\begin{proof}
Let $R \colon  \Ccal \rightarrow \Dcal$ be a right Quillen functor defining the localizer $S$.

One defines $\R^c_S \Ccal$ to have the same (anodyne cofibration,fibration) factorization system as $\Ccal$, and a cofibration of $\R^c_S \Ccal$ to be a cofibration of $\Ccal$ that have the left lifting property against all core fibrations that are in $S$, or equivalently all core fibrations $p$ such that $R(p)$ is acyclic. The fact that these forms a $\kappa$-accessible (resp. $\kappa$-combinatorial) weak factorization system on $\Ccal$ is proved using the exact dual of the argument used in the proof of \cref{prop:Bloc_at_class_inverted}. With this definition \ref{Th:RightBousfieldMain:same_fib} is immediate.

As $\R^c_S \Ccal$ has the same anodyne cofibrations and less cofibrations than $\Ccal$, \ref{Th:RightBousfieldMain:Id_Right_quillen} is immediate. For \ref{Th:RightBousfieldMain:it_inverts_arrows} let $s \in S$, and chose a core fibration representing $s$. It is sent by $R$ to an acyclic core fibration and hence an anodyne fibration in $\R^c_S \Ccal$ by definition. 

One then directly show property \ref{Th:RightBousfieldMain:is_universal}: let $R'\colon  \Ccal \rightarrow \Dcal'$ be a right Quillen functor inverting all arrows in $S$, with $\Dcal'$ core right saturated. One just need to show that $R'$ send all anodyne fibrations of $\R^c_S \Ccal$ to anodyne fibrations of $\Dcal'$, but as the anodyne fibrations of $\R^c_S \Ccal$ are generated (as a right class, i.e. cofibration are defined by the lifting properties against these maps) by the anodyne fibrations of $\Ccal$ and the core fibrations of $\Ccal$ that are in $S$, it is enough to show that this second class of arrows are sent to anodyne fibrations by $R'$. For any such map, its image by $R'$ is a core fibration and an equivalence, hence it is an acyclic core fibration, and hence by core right saturation of $\Dcal'$, it is an anodyne fibration. This concludes the proof of universality.

In particular, this universality property applies to the right Quillen functor $\Ccal \rightarrow \Dcal \rightarrow \R^c \Dcal$, hence we have a factorization $\Ccal \rightarrow \R^c_S \Ccal \rightarrow \R^c \Dcal$, so any arrow in $Ho(\Ccal)$ that is inverted by the right Quillen functor to $\R^c_S \Ccal$ also becomes an equivalence in $\R^c \Dcal$, which is the same as being an equivalence in $\Dcal$, hence it is in $S$. This proves the second half of \ref{Th:RightBousfieldMain:is_universal}.

We now prove point \ref{Th:RightBousfieldMain:is_core_saturated}: Given a core acyclic fibration of $\R^c_S \Ccal$, it is a core fibration of $\Ccal$, and an element of $S$ (we just proved that equivalences in $\R^c_S \Ccal$ are exactly the elements of $S$), so by definition of  $\R^c_S \Ccal$ it is an anodyne fibration of $\R^c_S \Ccal$. It remains to prove \ref{Th:RightBousfieldMain:kinda_same_cof}: but it follows from the dual of \cref{lem:nabla_lift} using the exact dual of the argument for \cref{Th:LeftBousfieldMain}.\ref{Th:LeftBousfieldMain:kinda_same_fib}.
\end{proof}

\begin{remark}\label{Rk:RightSat_Bousfield} Similarly to \cref{Rk:LeftSat_Bousfield}, one defines 

\[ \R_S \Ccal \colon = \R \R^c_S \Ccal \]

\noindent the right saturated Bousfield localization. It satisfies all the condition of \cref{th:RightBousfieldMain} where core right saturation is replaced by right saturation (in condition \ref{Th:RightBousfieldMain:is_core_saturated}, and in the universality property).
\end{remark}

\begin{remark} The exact dual of \cref{rk:Bousfield_on_Ho_cat} also applies here and $Ho(\R_S \Ccal)$ (equivalently $Ho(\R^c_S \Ccal)$) is identified with a coreflective full subcategory of $Ho(\Ccal)$ by the left Quillen functor $Id\colon  \R_S \Ccal \rightarrow \Ccal$.
\end{remark}

This construction of the right Bousfield localization we have given might look very different to these presents in the literature. We propose a slightly different construction to reduce this gap, we only formulate it in the combinatorial case to avoid some complications and make it closer to the classical literature.

\begin{definition} Given a combinatorial weak model category $\Ccal$, a pre-right Bousfield localization of $\Ccal$ is a combinatorial weak model structure on $\Ccal$ which:

\begin{itemize}

\item has the same fibrations as $\Ccal$;

\item has less cofibrations than $\Ccal$.

\end{itemize}

\end{definition}

\begin{prop}\label{prop:pre_right_Bloc} Let $\Ccal$ be a combinatorial weak model category. Let $I$ be a set of core cofibrations of $\Ccal$ such that for each $i \in I$ there exists a choice of $\nabla i$ as in \ref{lem:nabla_lift} which belongs the class generated by $I$ and the anodyne cofibrations of $\Ccal$. Then the class generated by $I$ and the anodyne cofibrations of $\Ccal$ are the cofibrations of a right saturated pre-right Bousfield localization of $\Ccal$. Moreover any right saturated combinatorial pre-right Bousfield localization of $\Ccal$ is obtained this way.
\end{prop}

\begin{proof}

Let $I$ be a set of core cofibrations as in the statement of the proposition, and let $\Rcal^I \Ccal$ be the combinatorial premodel structure on $\Ccal$ whose anodyne cofibrations are those of $\Ccal$ and whose cofibrations are generated by $I$ and the anodyne cofibrations of $\Ccal$. As the cofibrations of $\Rcal^I \Ccal$ are generated by core cofibrations and anodyne cofibrations it is immediate that it is right saturated. One just need to show that it is a weak model category to obtain that it is a pre-right localization. First one observes that as $\Rcal^I \Ccal$ has the same fibrations and more anodyne fibrations as $\Ccal$, any core fibration have a relative weak path object in $\Rcal^I \Ccal$, which is given by a relative weak path object taken in $\Ccal$. Moreover, it follows from \cref{lem:nabla_lift} that if $q, p$ are composable core fibrations such that $q$ and $q \circ p$ are acyclic (equivalently anodyne) fibrations then $p$ is also an acyclic fibration: indeed for each $i \in I$, as $\nabla i$ is a cofibration of $\Rcal^I \Ccal$ (for some choice of $\nabla i$), one has that $q$ has the lifting property against $\nabla i$ and $q \circ p$ has the lifting property against $i$, hence $p$ has the lifting property against $i$, and so by the dual form of \cref{prop:Model_str_from_3from2}, this concludes the proof that $\Ccal$ is a weak model category.

For the last part of the proposition: given $\Rcal \Ccal$ a right saturated pre-right Bousfield localization of $\Ccal$, then by right saturation its cofibrations are generated by anodyne cofibrations and core cofibrations. Let $I$ be a set of core cofibrations generating all the core cofibrations (it exists by \cref{lem:Constructing_some_WFS}.\ref{lem:Constructing_some_WFS:1}) then for any $i \in I$, $i$ has a relative cylinder objects in $\Rcal \Ccal$, but this is exactly the same as a choice of $\nabla i$ which is a cofibration of $\Rcal \Ccal$, and hence our set $I$ has the property claimed in the proposition.

\end{proof}

\begin{prop}\label{prop:right_loc_of_pre_right_loc} Let $\Rcal \Ccal$ a pre-right localization of $\Ccal$. Consider the right Bousfield localization $\R^c_S \Ccal$ of $\Ccal$ at the class $S$ of arrows inverted in $Ho(\Rcal \Ccal)$. Then $\Rcal \Ccal$ and  $\R^c_S \Ccal$ have the same right core, in particular the same homotopy theory.
\end{prop}

\begin{proof} By definition, both $\R^c_S \Ccal$ and $\Rcal \Ccal$ have the same (anodyne cofibrations, fibrations) factorization system as $\Ccal$. By construction, an acyclic\footnote{Or equivalently anodyne by core right saturation of $\Rcal^c_S \Ccal$.} core fibration of $\R^c_S \Ccal$ is exactly a core fibration of $\Ccal$ which is acyclic in $\Rcal \Ccal$, i.e. it is the same as an acyclic core fibration of $\Rcal \Ccal$, so $\R^c_S \Ccal$ and $\Rcal \Ccal$ have the same fibrations and core acyclic fibrations, that is they have the same right core. \end{proof}

\begin{construction} One can now describe the usual construction of a right Bousfield localization as found in the literature (we typically refer to the one in \cite{barwick2010left}): One starts with a class of cofibrant objects $K$, one choose iterated higher codiagonals $\nabla^i k$ for the core cofibrations $\emptyset \hookrightarrow k$ for $k \in K$ and we take $I$ to be the set of all these maps. This set $I$ clearly satisfies the condition of \cref{prop:pre_right_Bloc}, so it generates a pre-right Bousfield localization. We then take the right Bousfield localization inverting the same class of equivalences (which has the same right core, hence the same homotopy category by \cref{prop:right_loc_of_pre_right_loc}).
\end{construction}

Combining the result about Bousfield localizations of weak model categories, with the recognition principle for left and right semi-model categories proved in \cref{sec:WMS_with_cylinder} we obtain the following theorem:

\begin{theorem}\label{th:Bous_Loc_of_semi}

\begin{enumerate}

\item[]

\item\label{th:Bous_Loc_of_semi:left_L_left_S} Any left Bousfield localization $\Lb^c_S \Ccal$ of an accessible left semi-model category $\Ccal$ is also a left semi-model category.

\item\label{th:Bous_Loc_of_semi:right_L_right_S} Any right Bousfield localization $\R^c_S \Ccal$ of an accessible right semi-model category $\Ccal$ is also a right semi-model category.

\item\label{th:Bous_Loc_of_semi:left_L_right_S} Any left saturated left Bousfield localization $\Lb_S \Ccal$ of an accessible right semi-model category $\Ccal$ is also a right semi-model category.

\item\label{th:Bous_Loc_of_semi:right_L_left_S} Any right saturated right Bousfield localization $\R_S \Ccal$ of an accessible left semi-model category $\Ccal$ is also a left semi-model category.

\item\label{th:Bous_Loc_of_semi:two_sidded} If $\Ccal$ is an accessible two-sided model category (for e.g. an accessible Quillen model category), then each of its saturated left or right Bousfield localization $\R_S \Ccal$ and $\Lb_S \Ccal$ is a two-sided model category.

\end{enumerate}

\end{theorem}

The theorem works for both Fresse and Spitzweck's definition of left/right semi-model categories. In fact, if one uses Fresse's definition there is no longer a need to take saturated Bousfield localization $\Lb ,\R$ in point \ref{th:Bous_Loc_of_semi:left_L_right_S} and \ref{th:Bous_Loc_of_semi:right_L_left_S}: the ordinary unsaturated version $\Lb^c$ and $\R^c$ are enough.

It also follows that in point \ref{th:Bous_Loc_of_semi:left_L_left_S} and \ref{th:Bous_Loc_of_semi:right_L_right_S} a \emph{left saturated} left localization of an accessible left semi-model category is a left semi-model category, and dually.

\begin{proof}  For point \ref{th:Bous_Loc_of_semi:left_L_left_S}: one applies \cref{Th:recog_LMS}, $\Lb^c_S \Ccal$ is right saturated by \cref{lem:RightSat_LeftLoc} below because $\Ccal$ is, and core left saturated by definition of Bousfield localization. Given a cofibrant object $X \in \Lb^c_S \Ccal$, it is a cofibrant object of $\Ccal$, and a strong cylinder object for it in $\Ccal$ is still a strong cylinder object for it $\Lb^c_S \Ccal$ as (acyclic) cofibration in $\Ccal$ are still (acyclic) cofibration in  $\Lb^c_S \Ccal$. The proof of point \ref{th:Bous_Loc_of_semi:right_L_right_S} is the exact dual.  Point \ref{th:Bous_Loc_of_semi:left_L_right_S} is very similar: now $\Lb_S \Ccal$ is core right saturated by \cref{lem:RightSat_LeftLoc} because $\Ccal$ is, and given a fibrant objects $X$ in $\Lb_S \Ccal$, it is also fibrant in $\Ccal$, so as $\Ccal$ is assumed to be a right semi-model category it has a strong path object $PX$ in $\Ccal$. As $PX \rightarrow X$ is an anodyne fibration in $\Ccal$ (hence also in $\Lb_S \Ccal$), it follows that $PX$ is also fibrant in $\Lb_S \Ccal$, and hence applying point \ref{Th:LeftBousfieldMain:kinda_same_fib} of \cref{Th:LeftBousfieldMain} one concludes that $PX$ is again a strong path object in $\Lb_S \Ccal$. By \cref{Th:recog_RMS}, it follows that $\Lb_S \Ccal$ is a right semi-model category when $\Ccal$ is. The proof of point \ref{th:Bous_Loc_of_semi:right_L_left_S} is dual, and point \ref{th:Bous_Loc_of_semi:two_sidded} is a corollary of all the previous points combined.
\end{proof}

Point \ref{th:Bous_Loc_of_semi:right_L_right_S} is essentially C.~Barwick's result on right Bousfield localization proved in \cite{barwick2010left}, Point \ref{th:Bous_Loc_of_semi:left_L_left_S} is stated in \cite{barwick2010left} and its proof can be found in \cite{batanin2020Bousfield}, which also contains many interesting examples of left Bousfield localizations. Point \ref{th:Bous_Loc_of_semi:left_L_right_S},~\ref{th:Bous_Loc_of_semi:right_L_left_S} and \ref{th:Bous_Loc_of_semi:two_sidded} are new. Point \ref{th:Bous_Loc_of_semi:left_L_right_S} is probably the most interesting as left Bousfield localization are an extremely important tool in the theory of model categories and not having it for right semi-model categories was a huge drawback of the notion. Let me mention two applications of \ref{th:Bous_Loc_of_semi:left_L_right_S} that I have in mind:

\begin{itemize}

\item In \cite{Lanari2019homotopy} we needed to consider the left Bousfield localization of the right semi-model category of semi-simplicial sets which model homotopy $n$-type as an intermediate between Grothendieck $n$-groupoids and $n$-truncated spaces. The fact that this right Bousfield localization actually exists was left to be proved in the present paper.

\item In the author work on C.~Simpson semi-strictification conjecture (see \cite{henry2018regular}), all the model structures constructed on non-unital $\infty$-categories and positive polygraphs are either weak model structure or right semi-model structure. I'm hoping that these constructions will be extended to model for $(\infty,n)$-categories and $(\infty,\infty)$-categories in the future, and this requires many left Bousfield localizations, for example $(\infty,n)$-categories should be a left Bousfield localization of $(\infty,\infty)$-categories.
\end{itemize}

Point \ref{th:Bous_Loc_of_semi:right_L_left_S} is mostly here only for aesthetical purpose. I do not really have any concrete application in mind for Point \ref{th:Bous_Loc_of_semi:two_sidded} as the interest of the notion of two-sided model category is still unclear, but it does provide with a good supply of examples of two-sided model categories that are not automatically Quillen model categories, see \cref{ex:non_Quillen_Loc}.

We finish by the following lemma that was needed in the proof of our last theorem.

\begin{lemma}\label{lem:RightSat_LeftLoc} If $\Ccal$ is right saturated, or core right saturated, then its left Bousfield localization $\Lb^c_S \Ccal$ and $\Lb_S \Ccal $ also are. Dually if $\Ccal$ is left saturated, or core left saturated, then its right Bousfield localization $\R^c_S \Ccal$ and $\R_S \Ccal$ also are.
\end{lemma}

\begin{proof}
The proof of the two statement are exact duals, so we only treat left localizations. The case of $\Lb_S \Ccal$ follows from the case of $\Lb^c_S \Ccal$ and the fact (see \cref{th:main_saturation}) that the left saturation construction $\Lb$ preserves (core) right saturation. Now assume that $\Ccal$ is (core) right saturated. If $f$ is a (core) acyclic fibration in $\Lb^c_S \Ccal$, then it is also a (core) fibration in $\Ccal$, and as the two have the same cofibrations, it is still acyclic in $\Ccal$, hence by (core) right saturation of $\Ccal$ it is an anodyne fibration in $\Ccal$. As $\Ccal$ and  $\Lb^c_S \Ccal$ have the same cofibrations, they also have the same anodyne fibrations, so this concludes the proof.
\end{proof}

\appendix

\section{Locally $\kappa$-presentable categories}
\label{sec:lkp_cat}

The general goal of these two appendices is to record some results about combinatorial\footnote{Usually called ``cofibrantly generated''. We avoided this terminology as both accessible and combinatorial are defined by cofibrant generation conditions. See \cref{def:lambda_comb_wfs}.} and accessible weak factorization systems on locally presentable categories. Except the accessible case of \cref{th:appendix_Main}.\ref{th:appendix_Main:inf_wfs}, all the results presented here appear in some form in the literature, and we will try to give credit to the original source as much as we can. Often though the existing literature only contains these results in weaker form that do not specify the value of the accessibility rank (the cardinals $\lambda$ and $\kappa$) and here we will provide more precise version that do. In most cases, this only involves keeping track of the accessibility ranks through some already existing proof.

\Cref{sec:lkp_cat} regroup results about locally presentable categories and \cref{sec:apendix_awfs} regroup results about combinatorial and accessible weak factorization systems.

I am grateful to Ji\v{r}\'i Rosick\'y, John Bourke and Ivan Di Liberti for several helpful discussions related to the contents of these two appendices as well as for pointing out many references to me. The proof of the accessible case of \cref{th:appendix_Main}.\ref{th:appendix_Main:inf_wfs} below is due to John Bourke (personal communications), and Ji\v{r}\'i Rosick\'y also gave a different proof of the same result independently (also personal communications).

\begin{notation}\label{not:kappa_adjunction}
In what follows $\lambda$ and $\kappa$ always denote regular cardinals. We usually use $\kappa$ for an \emph{uncountable} regular cardinal and $\lambda$ for an arbitrary regular cardinal. We will say that an adjunction:
\[
\begin{tikzcd}
 \Ccal \ar[r,phantom,"\bot"description] \ar[r,bend left =35,"L"] & \Acal \ar[l,bend left =35,"R"] \\ 
\end{tikzcd}
\]
  is a \emph{$\lambda$-adjunction} if $\Ccal$ and $\Acal$ are both locally $\lambda$-presentable and if the following equivalent conditions are satisfied:

\begin{itemize}
\item The left adjoint functor $L$ preserves $\lambda$-presentable objects (i.e. is strongly $\lambda$-accessible).
\item The right adjoint functor $R$ preserves $\lambda$-filtered colimits (i.e. is $\lambda$-accessible).
\end{itemize}

We will also say that $L$ is a $\lambda$-left adjoint and that $R$ is a $\lambda$-right adjoint.

\end{notation}

\begin{theorem}\label{th:pres_rank} Let $\lambda$ and $\kappa$ be regular cardinals, with $\kappa$ uncountable.

\begin{enumerate}

\item\label{th:pres_rank:acc_limits} The bicategory of accessible categories and accessible functors between them has all small $\cat$-weighted pseudo-limits, and they are preserved by the forgetful functor to $\cat$.


\item\label{th:pres_rank:right_adj_limits} The bicategory of locally $\lambda$-presentable categories and $\lambda$-right adjoints between them has all small $\cat$-weighted pseudo-limits and they are preserved by the forgetful functor to $\cat$.


\item\label{th:pres_rank:left_adj_limits} The bicategory of locally $\kappa$-presentable category and $\kappa$-left adjoints between them has all $\kappa$-small $\cat$-weighted pseudo-limits and they are preserved by the forgetful functor to $\cat$.

\item\label{th:pres_rank:(co)algebras} Let $\Ccal$ be a locally $\kappa$-presentable category endowed with $T$ a $\kappa$-accessible endofunctor, (co)pointed endofunctor, or (co)monad. Then the category of $T$-(co)algebras is again locally $\kappa$-presentable and the adjunctions:

\[ L \colon  \Ccal \leftrightarrows T\text{-Alg} \colon  U \qquad U \colon  T\text{-Coalg} \leftrightarrows \Ccal \colon  R  \]
 
are $\kappa$-adjunctions. In the case of category of algebras this also works for $\kappa = \omega$.

\end{enumerate}
\end{theorem}

Most these results are well known and the proof will mostly point to references.

\begin{remark} \label{rk:PIE_limits} In \cref{th:pres_rank}, by ``$\cat$-weighted pseudo-limits'' we refer to the fact that as all the categories mentioned are bicategories, they can be considered as enriched (in an appropriate bicategorical sense) over the bicategory $\cat$ of categories, and hence we can consider the usual enriched notion of limits. These limits for bicategories have been considered by many authors under many different names\footnote{These notions have often be defined in the context of strict $2$-categories as special case of strict $2$-limits, and from this point of view are not equivalent. They are however all equivalent when considered in the bicategorical context.}: for example ``indexed limits of retract type'' in \cite{gregorybird}, or simply ``Limits'' in \cite{makkai1989accessible}, $2$-limits, bilimits, flexible limits, or PIE-limits. This last name comes from the fact that this class of limits is generated by Products, Inserters and Equifiers. Product is simply the usual notion of possibly infinity products. Given $F,G: \Ccal \rightrightarrows \Dcal$, the inserter $\text{Ins}(F,G) \rightarrow \Ccal$ is universal for $1$-arrows to $\Ccal$ endowed with a $2$-arrows $F \circ H \rightarrow G \circ H$, and given $\alpha,\beta: F \rightrightarrows G$ as before, the equifier $\text{Eq}(\alpha,\beta) \rightarrow \Ccal$ is universal for $1$-arrows $H$ to $\Ccal$ such that the $2$-arrow $\alpha \circ H, \alpha \circ H : F \circ H \rightrightarrows G \circ H$ are equals.

Such limits are called $\kappa$-small (for a regular cardinal $\kappa$) if both the indexing diagram is $\kappa$-small and the weight is levelwise $\kappa$-small. These $\kappa$-small $\cat$-weighted pseudo-limits are generated by $\kappa$-small products, inserters and equifiers. Pseudo-pullbacks and inverters are special cases of finite PIE or $\cat$-weighted pseudo-limits. \end{remark}

\begin{proof}[Proof of \cref{th:pres_rank}]
Point \ref{th:pres_rank:acc_limits} appears as Theorem 5.1.6 of \cite{makkai1989accessible} or in Section 2.H of \cite{adamek1994locally} for the case of lax limits, together with Exercise 2.m for inverters which together gives all weighted pseudo-limits. \ref{th:pres_rank:right_adj_limits} and \ref{th:pres_rank:left_adj_limits} are due to G.~Bird in his PhD thesis \cite{gregorybird}, as Theorem $2.17$ and Proposition $3.14$.

For \ref{th:pres_rank:(co)algebras} the usual argument goes as follows: in all cases, the category of (co)algebra can be written as a certain weighted limit of accessible categories and accessible functors between them (namely, $T$ and its powers), and hence it is accessible by \ref{th:pres_rank:acc_limits}. Moreover in categories of (co)algebra (co)limits are created by the forgetful functor to $\Ccal$, hence if $\Ccal$ is (co)complete, the category of (co)algebra is also (co)complete. In both case an accessible category which is either complete or cocomplete is locally presentable. But unfortunately, this argument gives no control over the presentability rank. We will need a different proof.

For the case of the category of algebras over a monad it is easy to show that the free algebra over $\kappa$-presentable objects form a dense family of $\kappa$-presentable algebras and hence that the category is indeed locally $\kappa$-presentable (even if $\kappa$ is countable). A detailed argument appears, for example, as Theorem 6.9 of \cite{gregorybird}. This also shows that the free algebra functor preserve $\kappa$-presentable objects, or equivalently it follows from the fact that as the monad preserve $\kappa$-filtered colimits, that the forgetful functor preserves $\kappa$-filtered colimits.

For categories of algebras over a $\kappa$-accessible endofunctor or pointed endofunctor one can simply observe that the forgetful functor $T$-Alg$ \rightarrow \Ccal$ is monadic,  as it satisfies all the assumption of Beck's monadicity criterion, and preserves $\kappa$-filtered colimits as $T$ does, because the forgetful functor $T$-Alg $\rightarrow \Ccal$ preserves all colimits that are preserved by $T$, and its left adjoint preserves all colimits. It follows that the category of $T$-algebras is the category of algebras for a $\kappa$-accessible monad, so that one can conclude the proof by using the previous case.  All of this applies without any problem even if $\kappa$ is countable.

We move to the case of coalgebras. The case of coalgebras over a $\kappa$-accessible endofunctor is treated in \cite{adamek2004tree} as Theorem 4.2, which more generally says that given a $\lambda$-accessible endofunctor $F$ on a $\lambda$-accessible category that admits colimits of $\omega$-chains, the category of $F$-coalgebras is $\lambda$-accessible and an $F$-coalgebra is $\lambda$-presentable if and only if its underlying object is $\lambda$-presentable.

Applying this to a locally $\lambda$-presentable category immediately gives that the category of coalgebras is locally $\lambda$-presentable and cocomplete (as the forgetful functor creates colimits) and hence is locally $\lambda$-presentable. Moreover the forgetful functor preserves (and detects) $\lambda$-presentable objects, so the adjunction is indeed a $\lambda$-adjunction.

It remains to prove the result in the case where $T$ is a copointed endofunctor or a comonad. Let $T=(T_0,\epsilon\colon T_0 \rightarrow \id)$ be a $\lambda$-accessible copointed endofunctor on a $\lambda$-presentable category $\Ccal$.

The category of $T$-coalgebras is the equifier:

\[
\begin{tikzpicture}[scale=1.5]
\node (M) at (-2,0) {$T$-Coalg};
\node (S) at (0,0) {$(T_0$-Coalg$)$};
\node (T) at (3,0) {$\Ccal$};

\draw[->] (M) to (S);

\draw[->,bend left=40] (S) to node[above]{$U$} node[below,name=Top]{$\quad$} (T);
\draw[->,bend right=40] (S) to node[below]{$U$} node[above,name=Bot]{$\quad$} (T);

\draw[double equal sign distance,->,-implies,bend left = 30,xshift=1cm] (Top.south east) to node[fill=white]{$\epsilon'$} (Bot.north east);
\draw[double equal sign distance,->,-implies,bend right = 30] (Top.south west) to node[fill=white]{$\id_U$} (Bot.north west);
\end{tikzpicture}
\]

where $U$ denotes the forgetful functor, and $\epsilon'$ is the natural transformation which on each $F$ algebra $(X,s:X \rightarrow F(X))$ is $\epsilon'_X = \epsilon_X \circ s : X \rightarrow X$. The $T$ algebras are the $T_0$-algebra $X$ such that $\epsilon'_X = \id$ which is just what this equifier captures. Hence the result follows from \ref{th:pres_rank:left_adj_limits} as $U$ is a left adjoint functor preserving $\kappa$-presentable objects. 

If now $T=(T_0,\epsilon,\eta)$ is a comonad, one has to adjust the argument a bit as this time we want to equify two natural transformations $X \rightrightarrows F^2(X)$, but $F^2(X)$ does not have to be a left adjoint functor preserving $\kappa$-presentable objects. We hence consider $\Ecal$ the category whose objects are triples, consisting of an object $Y \in \Ccal$ together with two arrows $Y \rightrightarrows T^2 Y$.  The category $\Ecal$ can be seen as the category of coalgebras for the $\kappa$-accessible endofunctor $Y \mapsto (T \circ T)(Y)^2$ and hence by the previous result is locally $\kappa$-presentable, with such a triple being $\kappa$-presentable if and only if $Y$ is.

We denote by $T_1$ the copointed endofunctor $(T_0,\epsilon)$.
There is a natural functor $Z\colon  T_1$-Coalg $\rightarrow \Ecal$ sending each $T_1$-coalgebra $(X,\nu)$ to the two morphisms $X \rightarrow T^2(X)$ whose equality witness that $X$ is a $T$-coalgebra.
$Z$ preserves all colimits and $\kappa$-presentable objects, hence is a $\kappa$-left adjoint.
Finally, we also have a $\kappa$-left adjoint functor $T^2$-Coalg $\rightarrow \Ecal$, which sends a coalgebra $\eta:X \rightarrow T^2(X)$ to $(X,\eta,\eta)$ and the category of $T$-Coalg appears as the pseudo-pullback:

\[
\begin{tikzcd}
  T\text{-Coalg} \ar[r] \ar[d] & T_1\text{-Coalg} \ar[d,"Z"] \\
T^2\text{-Coalg}  \ar[r] & \Ecal 
\end{tikzcd}
\]

And this concludes the proof by \ref{th:pres_rank:left_adj_limits}.

\end{proof}

\begin{lemma}\label{lem:monadic_limit} Let $\Ccal_i,\Dcal_i \colon  I \rightarrow \cat$  be two (pseudo) diagrams of categories, and let $V\colon \Ccal_i \rightarrow \Dcal_i$ be a pseudo-natural transformation which is levelwise monadic. Then the functor:

\[ \lim V\colon   \lim_{i \in I} \Ccal_i \rightarrow \lim_{i \in I} \Dcal_i, \]

\noindent where the limits are pseudo-limits weighted by some $I \rightarrow \cat$, is monadic if and only if it has a left adjoint.

\end{lemma}

\begin{proof}

We check that $\lim V$ always satisfies all the (others) conditions of the Beck monadicity criterion. Conservativity is clear. Given a $\lim V$-split pair in $\lim_{i \in I} \Ccal_i$, its projection to each $\Ccal_i$ (for each $i \in I$ and $x \in W(I)$) is a $V_i$-split pair, hence, as $V_i$ is monadic, it has a coequalizer in $\Ccal_i$ which is preserved by $V_i$. We claim that these coequalizers are preserved by all the transition maps $\Ccal_i \rightarrow \Ccal_j$ (for all maps in $I$ and in the weight), indeed in the naturality square:

\[
\begin{tikzcd}
  \Ccal_i \ar[r] \ar[d] & \Ccal_j \ar[d,"V_j"] \\ 
\Dcal_i \ar[r] & \Dcal_j \\
\end{tikzcd}
\]

The two vertical maps preserves these coequalizers, and the lower horizontal arrows also preserves them because in $\Dcal_i$ they are split coequalizers and hence preserved by all functors. It follows that natural map in $\Ccal_j$ between the coequalizer computed in $\Ccal_j$ and the image of the coequalizer computed in $\Ccal_i$ is sent by $U_j$ to an isomorphism, but as $U_j$ is conservative this proves that the coequalizer in $\Ccal_i$ is preserved by the transition map to $\Ccal_j$. It hence follows that there is a coequalizer in $\lim \Ccal_i$ and that it is preserved by all the functors to the $\Ccal_i$, hence by all the functor to $\Dcal_i$, and hence also by the functor to $\lim \Dcal_i$.
\end{proof}

\begin{cor}\label{cor:fiber_prod_comonadic} For any uncountable regular cardinal $\kappa$, a pseudo-pullback of a comonadic $\kappa$-left adjoint functor along a $\kappa$-left adjoint functor (between locally $\kappa$-presentable categories) is again a comonadic $\kappa$-left adjoint functor. Also, given a pseudo-pullback:

\[ \prod_\Acal \Ccal_i \]

of a family of comonadic $\kappa$-left adjoint functors $\Ccal_i \rightarrow \Acal$ (between locally $\kappa$-presentable categories), the functor $\prod_\Acal \Ccal_i \rightarrow \Acal$ is a comonadic $\kappa$-left adjoint.
\end{cor}

Note that in general composites of comonadic left adjoint functors are not always comonadic, so that the second part, even for finite families cannot be deduced directly from the first one.

\begin{proof}
In both cases, the existence of the limits and the fact that its structural maps are $\kappa$-left adjoints follows from Point \ref{th:pres_rank:left_adj_limits} of \cref{th:pres_rank}. Once we know that the functors involved are left adjoints, their comonadicity follows from the dual of \cref{lem:monadic_limit} applied to the natural transformation:

\[
\begin{tikzcd}
  \Bcal \ar[r] \ar[d] & \Acal \ar[d,equal ] & \Ccal \ar[l] \ar[d,equal] \\
\Acal \ar[r,equal] & \Acal & \Ccal \ar[l] \\
\end{tikzcd}
\]

where $\Bcal \rightarrow \Acal$ is comonadic, so that the resulting comparison map is $\Bcal \times_\Acal \Ccal \rightarrow \Ccal$. For the second part of the corollary, we apply \cref{lem:monadic_limit} to the natural transformation which is the identity on $\Acal$ and $\Ccal_i \rightarrow \Acal$ for all $i$.
\end{proof}

\section{Accessible and combinatorial weak factorization systems}
\label{sec:apendix_awfs}

Let $I$ be a set of arrows in a category $\Ccal$. We denote by $I^\wlp$ the class of arrows that have the right lifting property against $I$, and by $\cof(I)$ the retracts of transfinite composition of pushout of maps in $I$. We recall: 

\begin{prop}[Small object argument] Let $\Ccal$ be a co-complete category, and let $I$ be a set of arrows in $\Ccal$. Assume that $\Ccal$ is locally presentable, or more generally that for every $i\colon A \rightarrow B \in I$, the functor:

\[ \Hom(A, \_) \colon  \Ccal \rightarrow \set \]

commutes with $\lambda$-directed transfinite composites of pushouts of maps in $I$ for some regular cardinal $\lambda$. Then $(\cof(I),I^\wlp)$ is a weak factorization system on $\Ccal$.

\end{prop}

We now recall the algebraic small object argument, due to R.~Garner in \cite{garner2009understanding}. We consider this time a small category $I$ of arrows in $\Ccal$, i.e. a small category $I$ with a functor $I \rightarrow \Ccal^\rightarrow$ to the arrow category of $\Ccal$. One says that an arrow $f$ in $\Ccal$ has the algebraic lifting property against $I$, and write $f \in I^\alp$, if for each (solid) lifting problem as in the diagram on the left below, it is endowed with a chosen (dotted) diagonal filler:

\[
\begin{tikzcd}
  A \ar[d,"i \in I"swap] \ar[r,"u"] & X \ar[d,"f"] \\
B \ar[r,"v"swap] \ar[ur,dotted,"\tau(i{,u,}v)"description] & Y \\ 
\end{tikzcd} \qquad 
\begin{tikzcd}
A_1 \ar[d,"i_1"description] \ar[r,"k"] &  A_2 \ar[d,"i_2"description] \ar[r,"u"] & X \ar[d,"f"] \\
B_1 \ar[r,"h"] \ar[urr,dotted] & B_2 \ar[r,"v"swap] \ar[ur,dotted] & Y \\ 
\end{tikzcd}
\]

such that moreover, given a diagram as on the right above where the left square is the image of a morphism in $I$, the two diagonal fillers are compatible, that is $\tau(i_2,u,v) \circ h = \tau(i_1,uk,vh)$.

We denote by $I^\alp$ the category whose objects are the arrows in $\Ccal$ that are endowed with such compatible choice of lift, and whose morphisms are the squares compatible to these chosen lift. Though very often we also denote by $I^\alp$ the class of arrows that admits such a structure. Garner's version of the small object argument can be stated as:

\begin{prop}
Let $\Ccal$ be a locally presentable category and $I$ a small category of arrows in $\Ccal$. Then the forgetful functor $I^\alp \rightarrow \Ccal^\rightarrow$ is a monadic right adjoint functor. Moreover given a general arrow $f \in \Ccal$ the unit of the adjunction has the form:

\[
\begin{tikzcd}
 X \ar[d,"f"] \ar[r,"Lf"] & E \ar[d,"Rf"] \\
Y \ar[r,equal] & Y \\  
\end{tikzcd}
\]

where $Rf \in I^\alp$ and $Lf$ has the left lifting property against all arrows in the image of $I^\alp$.
\end{prop}

So in particular one obtains a weak factorization system whose right class are the arrows in the image of $I^\alp$ (this class is already closed under retract) and the left class are the arrow with the weak left lifting property against them. As we will recall very soon, this is in fact the underlying weak factorization system of an algebraic (or natural) weak factorization system as considered in \cite{grandis2006natural}, \cite{garner2009understanding}  and \cite{bourke2016algebraic}.

\begin{definition}\label{def:lambda_comb_wfs} A weak factorization system on a category $\Ccal$ is said to be \emph{$\lambda$-combinatorial} (resp. \emph{$\lambda$-accessible}) if:

\begin{itemize}

\item $\Ccal$ is locally $\lambda$-presentable.

\item Its right class is $I^\wlp$ (resp. $I^\alp$) for $I$ a set (resp. a small category) of arrows between $\lambda$-presentable objects of $\Ccal$.

\end{itemize}

It is said to be combinatorial (resp. accessible) if it is $\lambda$-combinatorial (resp. $\lambda$-accessible) for some $\lambda$.

\end{definition}

Usually, an accessible weak factorization system is defined in terms of the existence of an accessible factorization functor. \cref{th:acc_wfs_eq_def} below (mostly due to J.~Rosick\'y in \cite{rosicky2017accessible}) show that for an uncountable $\lambda$ this is equivalent to our definition. As previous references only defined the notion of accessible weak factorization system (and not $\lambda$-accessible) our terminology is not conflicting with pre-existing literature. The reason why we did not use the more usual definition, stated as condition \ref{th:acc_wfs_eq_def:accessible} of \cref{th:acc_wfs_eq_def}, is explained in \cref{rk:changin_the_def_of_accessible_wfs}.

\begin{remark}\label{Rk:accessible_and_algebraic_factorization}
By an \emph{accessible factorization functor}, we mean an accessible functor $(\Lb,\Eb,\Rb) \colon \Ccal^\rightarrow \rightarrow \Ccal^{\rightarrow \rightarrow}$ which sends every arrow $f \in \Ccal^\rightarrow$ to a factorization:

\[
\begin{tikzcd}
  X \ar[rr,"f"] \ar[dr,"\Lb(f)"swap] & & Y \\
 & \Eb (f)  \ar[ur,"\Rb(f)"swap] & \\ 
\end{tikzcd}
\]

The proof of the small object argument (in both versions) immediately show that the weak factorization system obtained from it admits such a functorial factorization. But it has been observed (see \cite{garner2009understanding}) that the factorization constructed from the small object argument have considerably more structure: they are ``\emph{algebraic weak factorization systems}'' (first introduced by Grandis and Tholen in \cite{grandis2006natural} under the name natural weak factorization systems, though the definition has been strengthened a bit by Garner in \cite{garner2009understanding}). This means essentially that the functors $\Lb$ and $\Rb \colon  \Ccal^\rightarrow \rightarrow \Ccal^\rightarrow$ are respectively a comonad and a monad that interact in a nice way. We refer to \cite{garner2009understanding} for the precise definition. We have: 
\end{remark}

\begin{theorem}\label{th:acc_wfs_eq_def} Let $\Ccal$ be a locally $\lambda$-presentable category, and $(L,R)$ a weak factorization system on $\Ccal$, consider the following conditions:

\begin{enumerate}[label=(\Alph*)]

\item\label{th:acc_wfs_eq_def:A_Garner_generated} $(L,R)$ is $\lambda$-accessible in the sense of \cref{def:lambda_comb_wfs}, i.e. it is generated under Garner's small object argument by a small category of arrows between $\lambda$-presentable objects.

\item\label{th:acc_wfs_eq_def:algebraic} $(L,R)$ is the underlying weak factorization system of a $\lambda$-accessible algebraic weak factorization system on $\Ccal$.

\item\label{th:acc_wfs_eq_def:accessible} $(L,R)$ admits a $\lambda$-accessible functorial factorization.

\end{enumerate}

Then \ref{th:acc_wfs_eq_def:A_Garner_generated} $\Rightarrow$ \ref{th:acc_wfs_eq_def:algebraic} $\Rightarrow$ \ref{th:acc_wfs_eq_def:accessible} and if $\lambda$ is uncountable all three conditions are equivalent.

\end{theorem}

This result is essentially due to J.~Rosick\'y as the main result of \cite{rosicky2017accessible}, which exactly states that condition \ref{th:acc_wfs_eq_def:A_Garner_generated} for some $\lambda$ is equivalent to condition \ref{th:acc_wfs_eq_def:accessible} for some (possibly different) $\lambda$, and the argument in the proof below is essentially the same as in \cite{rosicky2017accessible}.

I believe that in the case $\lambda = \omega$ both implications are strict, but I'm not aware of a known counter-example.

\begin{proof}

\ref{th:acc_wfs_eq_def:A_Garner_generated} $\Rightarrow$ \ref{th:acc_wfs_eq_def:algebraic} is essentially proved within the proof of Proposition 4.22 of \cite{garner2009understanding}:  The pointed endofunctor $L$ on $\Ccal^\rightarrow$ constructed there is such that the category of $L$-algebras is exactly $I^\alp$. They also show (within the proof) that if each arrow in $I$ is between $\lambda$-presentable object then $L$ is $\lambda$-accessible. By point (\ref{th:pres_rank:(co)algebras}) of \cref{th:pres_rank} it follows that the adjunction between $I^\alp$ and $\Ccal^\rightarrow$ is a $\lambda$-adjunction, and hence as the algebraic weak factorization system generated by $I$ is given by the unit of this adjunction, it is indeed a $\lambda$-accessible weak factorization system.

\smallskip

 \ref{th:acc_wfs_eq_def:algebraic} $\Rightarrow$ \ref{th:acc_wfs_eq_def:accessible} is clear.

\smallskip

It remains to show that when $\lambda$ is uncountable we have \ref{th:acc_wfs_eq_def:accessible} $\Rightarrow$ \ref{th:acc_wfs_eq_def:A_Garner_generated}, this is done by following the argument of \cite{rosicky2017accessible}, and we refer to \cite{rosicky2017accessible} for the details. We consider a $\lambda$-accessible functorial factorization for $(L,R)$, i.e. a $\lambda$-accessible functor $(\mathbb{L},\Eb,\mathbb{R})\colon \Ccal^\rightarrow \rightarrow \Ccal^{\rightarrow \rightarrow}$ as in \ref{Rk:accessible_and_algebraic_factorization}. We define the category of \emph{cloven $L$-maps} to be the category of coalgebras for the copointed endofunctor $\mathbb{L}\colon \Ccal^\rightarrow \rightarrow \Ccal^\rightarrow$, whose copointing is given by:

\[
\begin{tikzcd}
\bullet \ar[d,"\Lb f"] \ar[r,equal] & \bullet \ar[d,"f"] \\
\bullet \ar[r,"\Rb f"] & \bullet \\
\end{tikzcd}
\]

A standard retract argument shows that a map admits a structure of cloven $L$-map if and only if it is in $L$. By \cref{th:pres_rank}.\ref{th:pres_rank:(co)algebras}, as $\Lb$ is $\lambda$-accessible (and $\lambda$ is uncountable) the category of cloven $L$-maps, i.e. $\Lb$-Coalg, is locally $\lambda$-presentable, with the adjunction with $\Ccal^\rightarrow$ being a $\lambda$-adjunction. We let $I$ to be the essentially small full subcategory of $\Lb$-coalgebras of $\lambda$-presentable arrows, which is in particular a (essentially) small category of arrows between $\lambda$-presentable object in $\Ccal$. It follows from Lemma 24 of \cite{bourke2016algebraic} that $I^\alp = (\Lb\text{-Coalg})^\alp$. We claim that this implies that the weak factorization system generated by $I$ is $(L,R)$. Indeed any map in $I^\alp$ has the lifting property against all $\Lb$-coalgebras, i.e. all $L$-maps, and conversely any $R$-map admits a cloven $R$-map structure (i.e. is an algebra for the pointed endofunctor $\Rb$). But the lifting of a cloven $L$-map against a cloven $R$-map can be constructed explicitly in terms of the cloven structure and the functorial factorization, and hence is functorial in both morphisms of cloven $L$-maps and morphisms of cloven $R$-maps. It follows that any (cloven) $R$-map is in $I^\alp$.\end{proof}

\begin{remark}\label{rk:Bourke_Garner_Beck} Given an algebraic weak factorization system $(\Lb,\Rb)$ as mentioned in \cref{Rk:accessible_and_algebraic_factorization} on a category $\Ccal$, we have categories of $\Lb$-maps and $\Rb$-maps defined respectively as the categories of $\Lb$-coalgebras and $\Rb$-algebras. They come with functors $\Lb \rightarrow \Ccal^\rightarrow$ and $\Rb \rightarrow \Ccal^\rightarrow$. One of the main results of \cite{bourke2016algebraic} is a complete characterization of functors that arises this way. This characterization uses the concept of double categories (i.e. a category object in the category of categories). There is a double category of squares in $\Ccal$ whose category of objects is $\Ccal$ and whose category of morphisms is $\Ccal^\rightarrow$, and for any algebraic weak factorization system, one has a double categories whose category of object is $\Ccal$ and whose category of morphisms is the category of $\Lb$-maps (resp. $\Rb$-maps). The double categorical composition operation encodes the operations of composition of $\Lb$-maps and of $\Rb$-maps.

It is proved in \cite{bourke2016algebraic} (as Theorem 6, see also Theorem 3.6 of \cite{garner2020lifting}) that $\Lcal \rightarrow \Ccal^\rightarrow$ is the category of left arrows of an algebraic weak factorization system if and only it is the action on morphism of a double functor (morphism of double category) such that:

\begin{itemize}
\item Its action on objects is an isomorphism of categories.

\item Its action on morphisms, that is $\Lcal \rightarrow \Ccal^\rightarrow$, is (strictly) comonadic.

\item For any arrow $A \rightarrow B$ in $\Ccal$ that lifts to an object of $\Lcal$, there is a morphism to it in $\Lcal$ whose image in $\Ccal^\rightarrow$ is the square:

\[
\begin{tikzcd}
  A \ar[r,equal] \ar[d,equal] & A \ar[d,"f"] \\
 A \ar[r,"f"] & B \\
\end{tikzcd}
\]

\end{itemize}

The exact dual characterization of classes of right arrows also holds. Note that, for $\kappa$ uncountable, such an algebraic weak factorization system is $\kappa$-accessible as soon as the functor $\Lcal \rightarrow \Ccal^\rightarrow$ is a $\kappa$-left adjoint: indeed the factorization is given by the comonad associated to this left adjoint functor, so if its right adjoint is $\kappa$-accessible, hence the factorization is $\kappa$-accessible.

\end{remark}

\begin{notation}\label{nota:before_main_th}

\begin{itemize}

\item[]

\item Weak factorization systems on a category $\Ccal$ will be considered ordered by inclusion of their left classes. I.e. $(L,R) \leqslant (L',R')$ if $L \subset L'$ or equivalently if $R' \subset R$.

\item Given $(L_i,R_i)_{i\in I}$ a family of weak factorization systems on a category $\Ccal$ an infimum (resp. a supremum), if it exists, is a weak factorization system on $\Ccal$ whose left class (resp. right class) is $\cap L_i$ (resp. $\cap R_i$). These are indeed supremum and infimum in the sense of the order of the previous point.

\item Given $V\colon \Ccal \rightarrow \Dcal$ a (generally right, resp. left adjoint) functor and $(L,R)$ a weak factorization system on $\Dcal$. A right (resp. left) transfer of $(L,R)$ along $V$ is a weak factorization system on $\Ccal$ whose right (resp. left) class is $V^{-1}(R)$ (resp. $V^{-1}(L)$).

\item We define the following four bicategories:

 \[ \PC{R}{\lambda}{Comb} \quad \PC{R}{\lambda}{Acc} \quad \PC{L}{\lambda}{Comb} \quad \PC{L}{\lambda}{Acc} \]

whose objects are locally $\lambda$-presentable categories endowed with either $\lambda$-combinatorial or $\lambda$-accessible weak factorization systems and whose morphisms are either ``Quillen'' $\lambda$-right adjoint functor or ``Quillen'' $\lambda$-left adjoint functors, where a Quillen adjunction is an adjunction whose left adjoint preserves the left classes of the weak factorization system, or equivalently whose right adjoint preserves the right class of the weak factorization system.

\item We denote by $\cat_M$ the bicategory of categories endowed with a class of marked arrows (closed under isomorphisms) and functors preserving marked arrows. Each of the four categories introduced above has a forgetful functor to $\cat_M$ that send each category with a weak factorization system to the same category endowed with its left class for $\PC{L}{\lambda}{Comb}$ and $\PC{L}{\lambda}{Acc}$, and its right class for $\PC{R}{\lambda}{Comb}$ or $\PC{R}{\lambda}{Acc}$. These are functorial as left Quillen functor preserves left classes and right Quillen functor preserves right classes.

\end{itemize}

\end{notation}

\begin{theorem}\label{th:appendix_Main} Let $\kappa$ be an uncountable regular cardinal and $\lambda$ an arbitrary regular cardinal.

\begin{enumerate}

\item\label{th:appendix_Main:sup_wfs} The supremum of a family of $\lambda$-accessible (resp. $\lambda$-combinatorial) weak factorization systems exists and is $\lambda$-accessible (resp. $\lambda$-combinatorial).

\item\label{th:appendix_Main:inf_wfs} The infimum of a $\kappa$-small family of $\kappa$-accessible (resp. $\kappa$-combinatorial) weak factorization systems exists and is $\kappa$-accessible (resp. $\kappa$-combinatorial).

\item\label{th:appendix_Main:right_transfer} The right transfer of a $\lambda$-accessible (resp. $\lambda$-combinatorial) weak factorization system along a $\lambda$-right adjoint exists and is $\lambda$-accessible (resp. $\lambda$-combinatorial).

\item\label{th:appendix_Main:left_transfer} The left transfer of a $\kappa$-accessible (resp. $\kappa$-combinatorial) weak factorization system along a $\kappa$-left adjoint exists and is $\kappa$-accessible (resp. $\kappa$-combinatorial).

\item\label{th:appendix_Main:right_limits} The bicategories $\PC{R}{\lambda}{Comb}$ and $\PC{R}{\lambda}{Acc}$ have all small $\cat$-weighted pseudo-limits and they are preserved by the forgetful functor to $\cat_M$.

\item\label{th:appendix_Main:left_limits} The bicategories $\PC{L}{\kappa}{Comb}$ and $\PC{L}{\kappa}{Acc}$ have all $\kappa$-small $\cat$-weighted pseudo-limits and they are preserved by the forgetful functor to $\cat_M$.

\end{enumerate}

\end{theorem}

\begin{proof}

Point \ref{th:appendix_Main:sup_wfs} is easy: if each $(L_i,R_i)$ is cofibrantly generated by a set (resp. a small category) $I_i$ between $\lambda$-presentable objects then, $\coprod I_i$ is a set (resp. a small category) of maps between $\lambda$-presentable object and by the small object argument it generates a weak factorization system whose right class is $\cap R_i$.

\bigskip

Point \ref{th:appendix_Main:right_transfer} is also very classical: if $F\colon \Ccal \leftrightarrows \Dcal\colon U$ is a $\lambda$-adjunction and $(L,R)$ is cofibrantly generated by a set (resp. a small category) $I$ of arrows between $\lambda$-presentable objects then the $F(I)$ are generators for the right transferred weak factorization system on $\Dcal$, and as $F$ is a $\lambda$-left adjoint they form a set (resp. a small category) of arrows between $\lambda$-presentable objects.

\bigskip

Point \ref{th:appendix_Main:right_limits} follows immediately from these two results: given a diagram $(\Ccal_i,L_i,R_i)$ of $\lambda$-accessible (resp. $\lambda$-combinatorial) weak factorization systems with $\lambda$-right Quillen functors between them (and a weight $W(i) \in \cat$), one can first take the limit of the underlying diagram of locally $\lambda$-presentable categories to obtain a locally $\lambda$-presentable category $\Ccal$ with for each $i$ (and each object $x \in W(i)$) a $\lambda$-right Quillen functor $p_{i,x}\colon \Ccal \rightarrow \Ccal_i$ and one puts on $\Ccal$ the supremum $(L,R)$ of all the right transferred weak factorization system along the $p_{i,x}$, i.e. $f \in R$ if and only if $p_{i,x}(f) \in R_i$ for all $i$ and all $x \in W(i)$. This weak factorization system exists and is $\lambda$-accessible (resp. $\lambda$-combinatorial) by the previous results, and one easily check that this is the limit.

\bigskip

Point \ref{th:appendix_Main:left_limits} for combinatorial weak factorization systems is essentially the main result of \cite{makkai2014cellular}. Removing the restriction to a fixed cardinal $\kappa$, this is exactly what their Corollary $3.3$ claims. While they do not explicitly claim the result for a fixed cardinal $\kappa$, the proof of their Theorem 3.2 shows that $\PC{L}{\kappa}{Comb}$ has pseudo-pullback (computed in $\cat_M$), whilst the proof of Corollary 3.3 shows how equifiers and inserters can be deduced from pseudopullbacks. Finally their argument for the product of combinatorial weak factorization systems also shows that $\kappa$-small products of $\kappa$-combinatorial weak factorization systems are $\kappa$-combinatorial as an object $(X_i) \in \prod \Ccal_i$ is $\kappa$-presentable if and only if each of its component $X_i$ is $\kappa$-combinatorial (when the product is $\kappa$-small). All these facts together imply the result.

\bigskip

Point \ref{th:appendix_Main:left_transfer} and \ref{th:appendix_Main:inf_wfs} for combinatorial weak factorization system can be deduced from point \ref{th:appendix_Main:left_limits}: As mentioned as Remark 3.8 in \cite{makkai2014cellular}, a left transfer along $F\colon \Ccal \rightarrow \Dcal$ can be seen as a pullback:

\[
\begin{tikzcd}
  \Ccal \ar[d] \ar[r] \ar[dr,phantom,"\lrcorner"very near start] & \Dcal \ar[d] \\
\Ccal_t \ar[r] & \Dcal_t \\
\end{tikzcd}
\]

Where $\Ccal_t$ and $\Dcal_t$ denotes $\Ccal$ and $\Dcal$ endowed with the trivial factorization system (all map, isomorphism), and the $\Ccal$ on the top left corner comes equipped with the left transfer weak factorization system. As all the functors above are $\kappa$-left adjoint and the factorization system (All map, Iso) is $\kappa$-combinatorial (it is generated by a set of representative of all maps between $\kappa$-presentable objects) the left transfered weak factorization system on $\Ccal$ is also $\kappa$-combinatorial. Similarly, if $(L_i,R_i)$ is a $\kappa$-small set of $\kappa$-combinatorial weak factorization system on a locally $\kappa$-presentable category, then their infimum can be obtained as the limit of the diagram of all $\Ccal_i=(\Ccal,L_i,R_i)$ all endowed with a map to $\Ccal_t$. 

\bigskip

Point \ref{th:appendix_Main:left_transfer} (and \ref{th:appendix_Main:right_transfer}) for accessible weak factorization systems is the main results of \cite{garner2020lifting}. We gave a sketch of their argument to show that it indeed gives the correct accessibility rank, but ultimately we refer to \cite{garner2020lifting} for the details.

Let $U\colon \Dcal \rightarrow \Ccal$ be a $\kappa$-left adjoint functor and let $(L,R)$ be a $\kappa$-accessible weak factorization system on $\Ccal$. We choose $(\Lb,\Rb)$ a $\kappa$-accessible algebraic realization (as per Condition \ref{th:acc_wfs_eq_def:algebraic} of \cref{th:acc_wfs_eq_def}) of this weak factorization system. One can assume that each $L$-map admits an $\Lb$-algebra structure: indeed if not one can consider instead the algebraic weak factorization system $(\Lb^\sharp,\Rb^\sharp)$ defined in \cite{garner2020lifting} such that $\Lb^\sharp$-coalgebras are exactly the cloven $\Lb$-maps, that is the $\Lb$-algebras, but with $\Lb$ seen as a copointed endofunctor instead of a comonad. It follows from \cref{th:pres_rank}.\ref{th:pres_rank:(co)algebras} that the comonad $\Lb^\sharp$, and hence the algebraic weak factorization system $(\Lb^\sharp,\Rb^\sharp)$, is $\kappa$-accessible: The functor $\Lb$-Clov $\rightarrow \Ccal^\rightarrow$ is a $\kappa$-left adjoint by \cref{th:pres_rank}.\ref{th:pres_rank:(co)algebras} as it is the forgetful functor from the category of coalgebras for a $\kappa$-accessible copointed endofunctor, and hence the induced comonad is $\kappa$-accessible. We form the pseudo-pullback:

\[
\begin{tikzcd}
\Tcal \ar[r] \ar[d] \ar[dr,phantom,"\lrcorner"very near start]  & \Lb\text{-Coalg} \ar[d] \\
  \Dcal^\rightarrow \ar[r] & \Ccal^\rightarrow \\
\end{tikzcd}
\]

which is a limit of $\kappa$-left adjoint between locally $\kappa$-presentable categories. \Cref{th:pres_rank}.\ref{th:pres_rank:left_adj_limits} shows that $\Tcal$ is locally $\kappa$-presentable and the functor $\Tcal \rightarrow \Dcal^\rightarrow$ is a $\kappa$-left adjoint. 

In \cite{garner2020lifting} it is shown, using the characterization of algebraic weak factorization system as double categories discussed in \cref{rk:Bourke_Garner_Beck} that $\Tcal \rightarrow \Dcal^\rightarrow$ is the category of left arrows of an algebraic weak factorization system. Indeed $\Tcal$ comes with the double category structure over $\Dcal^\rightarrow$ induced (along the pullback) by that of $\Lb$-Coalg, it is a comonadic $\kappa$-left adjoint by \cref{cor:fiber_prod_comonadic}, and the third condition follows immediately from the fact that $\Lb$-Coalg satisfies it.  Hence $\Tcal$ is the category of left arrows of some $\kappa$-accessible algebraic weak factorization system on $\Dcal$. This immediately implies the results: the left class of the underlying weak factorization are exactly the arrows of $\Dcal$ whose image in $\Ccal$ admits a $\Lb$-algebra structure, i.e. are $L$-maps by our assumption. As the map $\Tcal \rightarrow \Dcal^\rightarrow$ is $\kappa$-left adjoint the corresponding algebraic weak factorization system in $\kappa$-accessible.

\bigskip

We move to point \ref{th:appendix_Main:inf_wfs} for accessible weak factorization system. The proof is very similar to the case of point \ref{th:appendix_Main:left_transfer}, but do not seem to appear in the literature. We learned the argument below from John Bourke (personal communication). Ji\v{r}\'i Rosick\'y also found and sent us a different proof at the same time. 

Given $(L_i,R_i)_{i \in I}$ a $\kappa$-small collection of $\kappa$-accessible weak factorization systems on a locally $\kappa$-presentable category $\Ccal$. For each $i$ we chose $(\Lb_i,\Rb_i)$ a $\kappa$-accessible algebraic realization such that each $L_i$-map admits an $\Lb_i$-coalgebra structure, exactly as done for point \ref{th:appendix_Main:left_transfer} above. For each $i$ one has a comonadic $\kappa$-left adjoint functor $\Lb_i$-Coalg $\rightarrow \Ccal^\rightarrow$ which is the action on morphisms of a double category functor satisfying all the conditions mentioned in \cref{rk:Bourke_Garner_Beck}. One then form the large pseudo-pullback:

 \[ \Lcal \coloneqq \prod_{\Ccal^\rightarrow} \left( \Lb_i\text{-Coalg} \right),\]
by \cref{cor:fiber_prod_comonadic} the resulting functor $\Lcal \rightarrow \Ccal^\rightarrow$ is a comonadic $\kappa$-left adjoint and it still satisfies the other conditions from \cref{rk:Bourke_Garner_Beck}, hence it is the category of left arrows of a $\kappa$-accessible algebraic weak factorization system. An arrow in $\Ccal$ lifts to an object of $\Lcal$ if and only if it admits a $\Lb_i$-coalgebra structure for all $i$, hence the underlying weak factorization system of this $\kappa$-accessible algebraic factorization is indeed the infimum of all the $(L_i,R_i)$.

\bigskip

Point \ref{th:appendix_Main:left_limits} for accessible weak factorization systems, i.e. for the category $\PC{L}{\kappa}{Acc}$ can be deduced from points \ref{th:appendix_Main:inf_wfs} and \ref{th:appendix_Main:left_transfer} and \cref{th:pres_rank}.\ref{th:pres_rank:left_adj_limits} in the exact dual way as we proved point \ref{th:appendix_Main:right_limits} for combinatorial weak factorization system from points \ref{th:appendix_Main:sup_wfs} and \ref{th:appendix_Main:right_transfer} above.

\end{proof}

\begin{remark}\label{rk:changin_the_def_of_accessible_wfs} The reason we chose point \ref{th:acc_wfs_eq_def:A_Garner_generated} of \cref{th:acc_wfs_eq_def} as our definition of $\lambda$-accessible weak factorization system (when $\lambda$ is countable) is because it gives very simple proofs for points \ref{th:appendix_Main:sup_wfs} and \ref{th:appendix_Main:right_transfer} of \cref{th:appendix_Main}. It should be noted however that both results still hold if we use point \ref{th:acc_wfs_eq_def:algebraic} of \cref{th:acc_wfs_eq_def} as our definition instead: this is proved using the dual version of the arguments we gave above for points \ref{th:appendix_Main:left_transfer} and \ref{th:appendix_Main:inf_wfs} in the accessible case. We have not looked for counter-example but we believe that on the contrary points \ref{th:appendix_Main:sup_wfs} and \ref{th:appendix_Main:right_transfer} do not hold for $\lambda= \omega$ if we use point \ref{th:acc_wfs_eq_def:accessible} of \cref{th:acc_wfs_eq_def} as our definition, which is the reason why we have not used it despite being apparently the most natural choice to make. Of course when $\lambda$ is uncountable all these definitions become equivalent by \cref{th:acc_wfs_eq_def}, and we no longer need this discussion.
\end{remark}

\section*{Acknowledgment}

I'm gratefull to Ji\v{r}\'i Rosick\'y, John Bourke and Ivan Di Liberti for some help regarding the appendix detailed there. I'm also gratefull to Michael Batanin and David White for sharing with me a preliminary version of \cite{batanin2020Bousfield}. I thanks David White for some helpful discusion on the notion of semi-model category and John Bourke for his comments on earlier draft of the paper, and for pointing out many typos.

\bibliography{/home/henry/Dropbox/Latex/Biblio}{}

\begin{thebibliography}{10}

\bibitem{adamek2004tree}
Jiri Ad{\'a}mek and H-E Porst.
\newblock On tree coalgebras and coalgebra presentations.
\newblock {\em Theoretical Computer Science}, 311(1-3):257--283, 2004.

\bibitem{adamek1994locally}
Ji{\v{r}}{\'\i} Ad{\'a}mek and Ji{\v{r}}{\'\i} Rosick{\'y}.
\newblock {\em Locally presentable and accessible categories}, volume 189.
\newblock Cambridge University Press, 1994.

\bibitem{325390}
Reid Barton.
\newblock Counter-example to the existence of left {B}ousfield localization of
  combinatorial model category.
\newblock MathOverflow.
\newblock \url{https://mathoverflow.net/q/325390} (version: 2019-03-13).

\bibitem{barton2020model}
Reid~W. Barton.
\newblock A model 2-category of enriched combinatorial premodel categories.
\newblock {\em Preprint arXiv:2004.12937}, 2020.

\bibitem{barwick2010left}
Clark Barwick.
\newblock On left and right model categories and left and right {B}ousfield
  localizations.
\newblock {\em Homology, Homotopy and Applications}, 12(2):245--320, 2010.

\bibitem{batanin2020Bousfield}
Michael Batanin and David White.
\newblock {L}eft {B}ousfield localization without left properness.
\newblock {\em Preprint ArXiv:2001.03764}, 2020.

\bibitem{gregorybird}
Gregory~J Bird.
\newblock Limits in 2-categories of locally presentable categories.
\newblock PhD thesis, University of Sydney, 1984. Circulated by the Sydney
  Category theory seminar.

\bibitem{bourke2016algebraic}
John Bourke and Richard Garner.
\newblock Algebraic weak factorisation systems {I}: {A}ccessible {AWFS}.
\newblock {\em Journal of Pure and Applied Algebra}, 220(1):108--147, 2016.

\bibitem{bourke2020fibrant}
John Bourke and Simon Henry.
\newblock Algebraically cofibrant and fibrant objects revisited.
\newblock {\em Preprint ArXiv:2005.05384}, 2020.

\bibitem{ching2014coalgebraic}
Michael Ching and Emily Riehl.
\newblock Coalgebraic models for combinatorial model categories.
\newblock {\em {Preprint ArXiv:1403.5303}}, 2014.

\bibitem{cisinski2002theories}
Denis-Charles Cisinski.
\newblock Th{\'e}ories homotopiques dans les topos.
\newblock {\em Journal of Pure and Applied Algebra}, 174(1):43--82, 2002.

\bibitem{cisinski2006prefaisceaux}
Denis-Charles Cisinski.
\newblock {\em Les pr{\'e}faisceaux comme mod{\`e}les des types d'homotopie}.
\newblock Soci{\'e}t{\'e} math{\'e}matique de France, 2006.

\bibitem{fresse2009modules}
Benoit Fresse.
\newblock {\em Modules over operads and functors}, volume 169 of {\em Lecture
  Notes in Mathematics}.
\newblock Springer-Verlag, Berlin, 2009.

\bibitem{garner2009understanding}
Richard Garner.
\newblock Understanding the small object argument.
\newblock {\em Applied categorical structures}, 17(3):247--285, 2009.

\bibitem{garner2020lifting}
Richard Garner, Magdalena Kedziorek, and Emily Riehl.
\newblock Lifting accessible model structures.
\newblock {\em Journal of Topology}, 13(1):59--76, 2020.

\bibitem{grandis2006natural}
Marco Grandis and Walter Tholen.
\newblock Natural weak factorization systems.
\newblock {\em Archivum Mathematicum}, 42(4):397--408, 2006.

\bibitem{henry2018regular}
Simon Henry.
\newblock Regular polygraphs and the {S}impson conjecture.
\newblock {\em {Preprint, ArXiv:1807.02627}}, 2018.

\bibitem{henry2018weakmodel}
Simon Henry.
\newblock Weak model categories in constructive and classical mathematics.
\newblock {\em {Preprint, ArXiv:1807.02650v3}}, 2018.

\bibitem{henry2020model}
Simon Henry.
\newblock The model categories of model categories.
\newblock {\em To appear}, 2021.

\bibitem{Lanari2019homotopy}
Simon Henry and Edoardo Lanari.
\newblock On the homotopy hypothesis in dimension 3.
\newblock {\em {Preprint ArXiv:1905.05625}}, 2019.

\bibitem{hovey1998monoidal}
Mark Hovey.
\newblock Monoidal model categories.
\newblock {\em Preprint arXiv:math/9803002}, 1998.

\bibitem{joyal2006quasi}
Andre Joyal and Myles Tierney.
\newblock {Q}uasi-categories vs {S}egal spaces.
\newblock {\em Contemporary Mathematics}, 431, 2007.

\bibitem{lafont2010folk}
Yves Lafont, Fran{\c{c}}ois M{\'e}tayer, and Krzysztof Worytkiewicz.
\newblock A folk model structure on omega-cat.
\newblock {\em Advances in Mathematics}, 224(3):1183--1231, 2010.

\bibitem{makkai1989accessible}
Michael Makkai and Robert Par{\'e}.
\newblock {\em Accessible categories: the foundations of categorical model
  theory}, volume 104.
\newblock American Mathematical Soc., 1989.

\bibitem{makkai2014cellular}
Michael Makkai and Ji\v{r}\'i Rosick{\'y}.
\newblock Cellular categories.
\newblock {\em Journal of Pure and Applied Algebra}, 218(9):1652--1664, 2014.

\bibitem{nikolaus2011algebraic}
Thomas Nikolaus.
\newblock Algebraic models for higher categories.
\newblock {\em Indagationes Mathematicae}, 21(1):52--75, 2011.

\bibitem{olschok2011left}
Marc Olschok.
\newblock Left determined model structures for locally presentable categories.
\newblock {\em Applied Categorical Structures}, 19(6):901--938, 2011.

\bibitem{rosicky2017accessible}
Jir{\'\i} Rosick{\`y}.
\newblock Accessible model categories.
\newblock {\em Applied Categorical Structures}, 25(2):187--196, 2017.

\bibitem{spitzweck2001operads}
Markus Spitzweck.
\newblock {\em Operads, algebras and modules in model categories and motives}.
\newblock PhD thesis, Ph. D. thesis (Universit{\"a}t Bonn), 2001.

\end{thebibliography}
\bibliographystyle{plain}

\end{document}